\documentclass[12pt]{amsart}

\title{Noncommutative linear systems and noncommutative elliptic curves}

\author{Daniel Chan}
\address{School of Mathematics and Statistics,
UNSW Sydney,
NSW 2052,
Australia
}
\email{danielc@unsw.edu.au}

\author{Adam Nyman}
\address{Western Washington University, Bellingham, USA
}
\email{nymana@wwu.edu}

\usepackage{mymacros,amssymb}
\usepackage{amsmath}
\usepackage{amsfonts}
\usepackage{amscd}

\newcommand{\usHomc}{{\underline{{\mathcal{H}}{\it{om}}}}_{C}}

\begin{document}
\maketitle

\begin{abstract}
In this paper we introduce a noncommutative analogue of the notion of linear system, which we call a helix $\underline{\mathcal{L}} := (\mathcal{L}_{i})_{i \in \mathbb{Z}}$ in an abelian category ${\sf C}$ over a quadratic $\bZ$-indexed algebra $A$. We show that, under natural hypotheses, a helix induces a morphism of noncommutative spaces from {\sf Proj}\,$\End(\underline{\cL})$ to {\sf Proj}\,~$A$. We construct examples of helices of vector bundles on elliptic curves generalizing the elliptic helices of line bundles constructed by Bondal-Polishchuk, where $A$ is the quadratic part of $B:= \End(\underline{\cL})$. In this case, we identify $B$ as the quotient of the Koszul algebra $A$ by a normal family of regular elements of degree 3, and show that {\sf Proj}\,~$B$ is a noncommutative elliptic curve in the sense of Polishchuk \cite{polish1}. One interprets this as embedding the noncommutative elliptic curve as a cubic divisor in some noncommutative projective plane, hence generalizing some well-known results of Artin-Tate-Van den Bergh.
\end{abstract}

\section{Introduction}

In classical algebraic geometry, one often uses sections of a line bundle $\cL$ on a projective variety $X$ over a field $k$ to construct a morphism $f \colon X \to \bP^{n}_k$. More algebraically, one obtains from this setup a homogeneous coordinate ring $B = \oplus_i H^0(X, f^*\cO(i))$ and a graded algebra homomorphism $k[x_0, \ldots, x_n] \to B$ which gives a Stein factorisation of $f$. Artin-Tate-Van den Bergh \cite{atv1} famously constructed a noncommutative example of this, ``embedding'' a genus one curve into a noncommutative projective plane.  They then used this embedding to study the noncommutative projective plane. We continue this line of study in this manuscript.

More precisely, in \cite{atv1}, the notion of line bundle is replaced by that of invertible bimodule, which are used to construct twisted homogeneous coordinate rings.  On the other hand, the notion of polynomial ring is replaced by that of  Artin-Schelter regular algebra. In \cite{bp}, a more flexible approach was adopted, where instead of taking powers of a single invertible bimodule, one considered a whole sequence $\underline{\cL} = (\cL_i)_{i \in \bZ}$ of line bundles on a genus one curve as analogues of $f^*\cO(i)$. The homogeneous coordinate ring is now replaced by the endomorphism algebra of $\underline{\cL}$, which is now a $\bZ$-indexed algebra as opposed to a $\bZ$-graded one. Their analogue of the polynomial ring is now a type of Koszul $\mathbb{Z}$-indexed algebra.

We further develop the approach in \cite{bp} as follows. We first extend the notion of Koszul algebra to the bimodule algebra setting and introduce the notion of Koszul complex in an additive category. As with the classical Koszul algebra, indexed bimodule Koszul algebras have exact Koszul complexes which are analogues of the Euler exact sequences for projective spaces. Our point of view is that if $f$ is some morphism of noncommutative spaces (whatever that might mean), then $f^*$ should preserve exact sequences of vector bundles. Our central notion, that of a helix, will be an analogue of the pullbacks of the Euler exact sequences on projective space. Given an abelian category {\sf C} and a quadratic $\bZ$-indexed algebra $A$, a {\it complete helix in {\sf C} of length n over $A$} (defined precisely in Definition~\ref{def:helix}), is a sequence $\underline{\cL} = (\cL_{i})_{i \in \bZ}$ of objects in {\sf C} enriched by a collection of exact sequences of the form
$$
0 \to \cL_{-j-n}^{\oplus m_{n,j}} \to \ldots \to \cL_{-j-2}^{\oplus m_{2,j}} \xto{\phi_2}
\cL_{-j-1}^{\oplus m_{1,j}} \xto{\phi_1}
\cL_{-j}^{\oplus m_{0,j}} \to 0
$$
for each $j \in \bZ$ such that the end terms involving $\phi_1,\phi_2$ resemble the Koszul complex over $A$. If $\cL$ is any sequence of objects, then the standard choice for $A$ is the quadratic part $\bS^{nc}(\underline{\cL}) := T(\underline{V})/\langle I_2\rangle$ of $\End(\underline{\cL})$ defined as the quotient of the tensor algebra on the degree one part $\underline{V}$ of $\End(\underline{\cL})$ quotiented out by the space of degree two relations $I_2$ of $\End(\underline{\cL})$. Before stating the following result, which gives our noncommutative version of a Stein factorization map and is a combination of Theorems~\ref{theorem.morph} and \ref{thm:adjoint}, we remind the reader that if $C$ is a $\mathbb{Z}$-indexed algebra, ${\sf Proj }C$ denotes the quotient of the category of graded right $C$-modules by a suitably defined torsion subcategory, generalizing the category of quasicoherent sheaves over a projective variety (see Section \ref{sec.geom} for a precise definition):
\begin{theorem}  \label{thm:mainStein}
Let $A$ be a quadratic algebra and let $\underline{\cL}$ be a complete length $n$ helix over $A$. Suppose, furthermore, that $B := \End(\underline{\cL})_{\geq 0}$ is locally finite and that there exists $m \geq 0$ such that for all $l \geq m$ and  $j>0$ we have $\Ext^j(\cL_i,\cL_{i+1}) = 0$. Then the restriction functor ${\sf Gr}\,B \to {\sf Gr}\, A$ induces a functor ${\sf Proj}\, B \to {\sf Proj}\, A$ which has a left adjoint. \end{theorem}

Various related notions of helices occur in the literature and inspired our definition. In particular, we will use a variant of Bondal-Polishchuk's notion of an elliptic helix $\underline{\cL}$ of period 3 in which the helix is composed of objects in a $k$-linear abelian category {\sf C}. The full definition is given in Definition~\ref{def:ellipticHelix}, but the key property here is that for all $i \in \bZ$, if one left mutates $\cL_i$ twice, first past $\cL_{i-1}$ and then that past $\cL_{i-2}$, one obtains $\cL_{i-3}$ up to isomorphism (see Definition \ref{def:mutable} for our notion of {\it mutation}). Each mutation corresponds to an exact sequence and splicing these together gives a complete helix of length 3 over $\bS^{nc}(\underline{\cL})$ as we show in Proposition~\ref{prop:canonicalHelixOfElliptic}.

The elliptic helices $\underline{\cL} = (\cL_i)$ of period 3 constructed by Bondal-Polishchuk involve line bundles $\cL_i$ on a smooth elliptic curve and $\End(\underline{\cL})$ is the analogue of Artin-Tate-Van den Bergh's twisted homogeneous coordinate ring which appears as the quotient of an Artin-Schelter regular algebra $A$ of dimension three, by the ideal generated by a normal element of degree three. We have the following analogue of this result. See Theorem~\ref{thm:BPBisKoszul}, Proposition~\ref{prop:HilbertSeries} and Theorem~\ref{thm:kernelGenByNormal} for precise statements and further details.

\begin{theorem}  \label{thm:mainSncEndL}
Let $\underline{\cL}$ be an elliptic helix of period 3 in $\Coh(X)$ where $X$ is a smooth elliptic curve over an algebraically closed field $k$ of characteristic zero. Let $A := \bS^{nc}(\underline{\cL})$ and suppose $B:= \End(\underline{\cL})$ is {\em equigenerated} in the sense that $\dim_k B_{01} = \dim_k B_{12} = \dim_k B_{23} = d$. Then
\begin{enumerate}
    \item $A$ is Koszul, 3-periodic, and is AS-regular of dimension three.
    \item the canonical map $A \to B$ is surjective and induces adjoint functors on {\sf Proj} as in Theorem~\ref{thm:mainStein}.
    \item the Hilbert series of $A$ and $B$ are
    $$
    H_{A,}(t) = \frac{1}{1 - dt + dt^2 - t^3}, \quad H_{B,}(t) = \frac{1-t^3}{1 - dt + dt^2 - t^3}
    $$
    \item the kernel of $A \to B$ is generated by a normal family $\underline{g}$ of regular elements of degree three.
\end{enumerate}
\end{theorem}

Finally, we produce a new family of examples of elliptic helices of period 3, consisting now of vector bundles on a smooth elliptic curve. The following sums up Theorem~\ref{thm:egTriadGenHelix} and Theorem~\ref{thm:ncElliptic}.
\begin{theorem}   \label{thm:mainNCelliptic}
Let $X$ be a smooth elliptic curve over an algebraically closed field $k$ of characteristic zero, and let $d>3$ an odd integer. Given any two line bundles $\cL_0, \cL_1$ of degrees 0 and $d$ and a rank two vector bundle $\cL'_1$ of degree $d$, there exists a unique elliptic helix $\underline{\cL} = (\cL_i)$ of period 3, incorporating the line bundles $\cL_0,\cL_1$ above and such that $\cL_2$ is the right mutation $R_{\cL_1} \cL'_1$. In this case, $B:= \End(\underline{\cL})$ is equigenerated with $\dim_k B_{01} = d$ so that Theorem~\ref{thm:mainSncEndL} applies. Furthermore, $B$ is coherent but not noetherian and {\sf Proj}\,$B$ is a noncommutative elliptic curve in the sense of \cite{polish1}.
\end{theorem}
More generally, we give in Theorem~\ref{thm:numericalCriterionGenHelix} a numerical criterion for when a triple of vector bundles $\cL_0,\cL_1,\cL_2$ on a smooth elliptic curve can be completed to an elliptic helix of period three. It is unfortunately, not so easy to check as it is in terms of positivity of a recursively defined sequence. The proof of this latter theorem is based on a criterion for being able to mutate vector bundles on elliptic curves given in Theorem~\ref{thm:evalDichotomy}.

\begin{notation} \label{notn:basedivrings}
Throughout, we let $(D_{i})_{i \in \mathbb{Z}}$ denote a sequence of division rings which we consider as ``base'' division rings. For many applications, all the $D_i$ will equal some fixed field $k$ which acts centrally on all objects. We let $D = \oplus D_i$ which we consider as a $\bZ$-indexed algebra concentrated in degree 0, that is, $D_{ii} = D_i$ and $D_{ij} = 0$ for $i \neq j$. By default, $D$-modules will be {\it left} $D$-modules unless otherwise stated, so a {\it $D$-module $\underline{\cL} = \oplus \cL_i$} in an additive category {\sf C} is just a collection of $D_i$-objects $\cL_i$ in ${\sf C}$, i.e. pairs $(\cL_i, \rho_i)$, where $\cL_i$ is an object in ${\sf C}$ and $\rho_i:D_i \to \End_{\sf C} \cL_{-i}$ is a ring homomorphism (the indexing convention is explained in Section \ref{sec.Background}). In practice, we will often start with a sequence of objects $\cL_i$ in {\sf C} whose endomorphism rings $\End \cL_i$ are division rings so we define $D_{-i} = \End \cL_i$, and so obtain a $D$-module in {\sf C}.

Given $M,N \in {\sf C}$, we abbreviate
$$
(M,N):= \Hom_{\sf C}(M,N)
$$
and
$$
^j(M,N) := \Ext^j_{\sf C}(M,N).
$$

If $X$ is a scheme, we let $\mbox{Coh}(X)$ denote the category of coherent sheaves over $X$.
\end{notation}

\section{Indexed bimodule algebras and Koszul theory}  \label{sec.Background}

In this section, we generalise the notion of Koszul $\bZ$-indexed $k$-algebra from \cite{bp} to the bimodule algebra setting.

Let $(D_{i})_{i \in \mathbb{Z}}$ denote a sequence of division rings and let $D = \oplus D_i$ as in Notation \ref{notn:basedivrings}.  A {\it $\bZ$-indexed $D$-algebra} is a ring $A$ with decomposition $A = \oplus_{i,j \in \bZ} A_{ij}$ such that
\begin{itemize}
\item{} $A_{ij}$ are $(D_i,D_j)$-bimodules,
\item{} multiplication is induced by associative multiplication maps $A_{ij} \otimes_{D_j} A_{jl} \to A_{jl}$ (multiplication $A_{ij} A_{kl}=0$ if $j \neq k$), and
\item{} each $A_{ii}$ contains a unit element $e_i$ satisfying the usual unit axiom.
\end{itemize}


Note the indexing convention differs from \cite{bp}. The {\it degree} of $A_{ij}$ and its elements is defined to be $j-i$. We will usually drop subscripts on $\otimes$ when we take tensor products of objects with indices: for example, $A_{ij} \otimes A_{jl} = A_{ij} \otimes_{D_j} A_{jl}$. This should not cause confusion since the default subscript, $D_i$, is determined by the index. Note that $D = \oplus D_i$ is a $\bZ$-indexed $D$-algebra with $D_{ii} = D_i$ and other summands zero.

Just as $k$-algebras often arise as endomorphisms of an object in a $k$-linear category {\sf C}, indexed $D$-algebras often come from a $D$-module $\underline{\mathcal{L}} := (\mathcal{L}_{i})_{i \in \mathbb{Z}}$ in an additive category {\sf C}. Indeed, such a module gives rise to a $\bZ$-indexed $D$-algebra $\textup{End}(\underline{\cL})$ defined as follows:
$$
\textup{End}(\underline{\cL})_{ij} := \Hom_{\sf C}(\cL_{-j},\cL_{-i})
$$
with multiplication defined as composition of homomorphisms. We choose our indexing convention for the following reasons: firstly, the order of indices is a result of our natural indexing on $\bZ$-indexed algebras. Secondly, the negative signs appear because we would like to view the $\cL_i$'s as analogues of line bundles $\cO(i)$ on a projective variety, and this convention keeps our $\bZ$-indexed algebra positively graded as a opposed to negatively graded. In addition, in the graded case, the sequence $\underline{\cL}$ typically comes from applying the orbit of an auto-equivalence to some object, so the endomorphism algebra is also called the {\it orbit algebra}.

Suppose now that $A$ is {\em locally finite} in the sense that all the $A_{ij}$ are finite dimensional on the left and right. Since {\sf C} is additive and $\cL_{-j}$ is a $D_{j}$-object in {\sf C}, $A_{ij} \otimes \cL_{-j}$ is a well-defined $D_i$-object in {\sf C} (see \cite[Section B3]{az3} for a concrete description of this tensor product in the context of an abelian category).  The tensor product is bifunctorial with respect to maps of finite-dimensional right $D_i$-modules and morphisms of $D_i$-objects, i.e. morphisms in ${\sf C}$ compatible with the $D_i$-action.  More precisely, versions of  \cite[Lemma B3.3, Lemma B3.9]{az3} hold in this context, and we will use this fact without comment in the sequel.

As one might expect from the non-indexed case, morphisms $A \to \End(\underline{\cL})$, correspond to left $A$-module structures on $\underline{\cL}$. The latter is a collection of multiplication maps
$$
\mu_{ij} \colon A_{ij} \otimes \cL_{-j} \longrightarrow \cL_{-i}
$$
satisfying the usual unit and associativity axioms of a module. The maps $\mu_{ij}$ induce morphisms $A_{ij} \to \End(\underline{\cL})_{ij}$ which define a morphism of $\bZ$-indexed algebras. Conversely, $\underline{\cL}$ is naturally a left $\End(\underline{\cL})$-module, so any morphism $A \to \End(\underline{\cL})$ defines an $A$-module structure on $\underline{\cL}$.

We turn now to generalising Koszul theory to the bimodule setting, taking our cue from \cite{bp}. One added complication is that given a $(D_i,D_j)$-bimodule $V$, we have both a left dual $^*V := \Hom_{D_i}(V,D_i)$ and a right dual $V^* := \Hom_{D_j}(V,D_j)$.
\begin{definition}  \label{def:quadratic}
A {\em quadratic $D$-algebra} is a locally finite $\bZ$-indexed $D$-algebra $A$, which is
\begin{enumerate}
    \item {\em positively graded} in the sense that $A_{ij} = 0$ for $i>j$,
    \item {\em connected} in the sense that all the $A_{ii} = D_i$
    \item {\em generated in degree one}, that is, by $\{A_{i, i+1}\}_{i \in \mathbb{Z}}$
    \item and has only quadratic relations $I_{i,i+2} \subset A_{i, i+1} \otimes A_{i+1, i+2}$ (see \cite{morinyman} for relevant definitions).
\end{enumerate}
In this case, we may define the {\em left Koszul dual} to be the negatively graded $\bZ$-indexed $D$-algebra $^!A$ generated in degree -1 by  $^!A_{i+1,i} := \ \ ^*A_{i,i+1}$ with relations $^{\bot}I_{i+2,i} \subset \ ^!A_{i+2,i+1} \otimes \ ^!A_{i+1,i}$ defined as the kernel of the map
$$
^!A_{i+2,i+1}\otimes \ ^!A_{i+1,i} = \ ^*A_{i+1,i+2} \otimes \ ^*A_{i,i+1} \cong \ ^*(A_{i,i+1} \otimes A_{i+1,i+2}) \rightarrow \ ^*I_{i,i+2}
$$
induced by the inclusion of relations.
\end{definition}
 We may similarly define the {\em right Koszul dual} $A^{!}$ and note that $(\ ^!A)^{!} \cong A$ as $^!A_{i+1,i}=\ ^*A_{i,i+1}$, and $(\,^{\bot}I_{i+2,i})^{\bot}$ is canonically isomorphic to $I_{i,i+2}$.

If $n \in \mathbb{Z}$ and $A$ is a $\mathbb{Z}$-indexed $D$-algebra, we let $A(n)$ denote the induced $\mathbb{Z}$-indexed $D$-algebra with $A(n)_{ij}=A_{i+n,j+n}$.  Following \cite[Section 2]{vdbquad}, we say $A$ is {\it $n$-periodic} if there is an isomorphism of $\mathbb{Z}$-indexed algebras $A \rightarrow A(n)$.

\begin{lemma} \label{lemma.Periodic}
If $A$ is quadratic and $^!A$ is $n$-periodic, then $A$ is $n$-periodic.
\end{lemma}

\begin{proof}
By hypothesis, there exist isomorphisms $^!A_{i+1,i} \longrightarrow  \ ^!A_{i+n+1,i+n}$ such that there is an induced commutative diagram with vertical isomorphisms and with right horizontals equal to multiplication
$$
\begin{CD}
0 & \longrightarrow & ^{\bot}I_{i+2,i} & \longrightarrow & ^!A_{i+2,i+1} \otimes \ ^!A_{i+1,i} & \longrightarrow & ^!A_{i+2,i} & \longrightarrow & 0 \\
& & @VVV @VVV @VVV \\
0 & \longrightarrow & ^{\bot}I_{i+n+2,i+n} & \longrightarrow & ^!A_{i+n+2,i+n+1} \otimes \ ^!A_{i+n+1,i+n} & \longrightarrow & ^!A_{i+n+2,i+n} & \longrightarrow & 0 \\
\end{CD}
$$
This induces a commutative diagram with vertical isomorphisms
$$
\begin{CD}
0 & \longrightarrow & ^{\bot}I_{i+2,i} & \longrightarrow & \,^*(A_{i,i+1} \otimes A_{i+1,i+2}) & \longrightarrow & ^*I_{i,i+2} & \longrightarrow & 0 \\
& & @VVV @VVV @VVV \\
0 & \longrightarrow & ^{\bot}I_{i+n+2,i+n} & \longrightarrow & \,^*(A_{i+n,i+n+1}\otimes A_{i+n+1,i+n+2}) & \longrightarrow & \,^*I_{i+n,i+n+2} & \longrightarrow & 0 \\
\end{CD}
$$
where the right horizontal maps are induced by inclusion.  By taking duals, we see that the isomorphisms in degree one are compatible with the relations in $A$, whence the result.
\end{proof}

Before introducing a fairly general notion of a Koszul complex, we present some motivational arguments. Let $\underline{\cL} = (\cL_i)$ be a $D$-module in an additive category {\sf C} and let $V_i = \Hom_{\sf C}(\cL_{-i-1}, \cL_{-i})$. If the $V_i$ are finite dimensional on both sides, so that the tensor algebra $T(\underline{V})$ over $D$ is locally finite, then the canonical morphism $T(\underline{V}) \to \End(\underline{\cL})$ makes $\underline{\cL}$ a left  $T(\underline{V})$-module.  {\it We will assume, throughout the rest of the paper, that $\underline{\cL}$ is such that $\End(\mathcal{L}_{i})$ and $V_{i}$ are finite dimensional on both sides for all $i \in \mathbb{Z}$.}

The following is a simple consequence of adjoint properties of tensor products of bimodules (see \cite[Section 2.1]{witt} for more details), and will be employed in Section \ref{sec.MapToEndo}:

\begin{lemma}   \label{lemma.factorViaUnit}
Let $W_i, W_{i+1}$ be finite dimensional right modules over $D_i, D_{i+1}$ respectively, so that $W_i \otimes \cL_{-i}, W_{i+1} \otimes \cL_{-i-1}$ are well-defined objects in {\sf C}. Consider a morphism
$\phi\colon W_{i+1} \otimes \cL_{-i-1} \to  W_i \otimes \cL_{-i}$ and the  associated composite map of $D_{i+1}$-spaces
$$ \delta \colon W_{i+1} = W_{i+1} \otimes D_{i+1} \to W_{i+1} \otimes  \End_{\sf C}(\cL_{-i-1}) \xto{\Hom(\cL_{-i-1},\phi)} W_i \otimes V_i.$$
Then $\phi$ factors in the following two ways.
\begin{enumerate}
    \item $\phi\colon W_{i+1} \otimes \cL_{-i-1} \xto{\delta \otimes 1} W_i \otimes V_i \otimes \cL_{-i-1} \xto{1 \otimes \mu} W_i \otimes \cL_{-i}
    $ \\
    where $\mu$ is module multiplication.
    \item $\phi \colon W_{i+1} \otimes \cL_{-i-1} \xto{1 \otimes \eta \otimes 1} W_{i+1} \otimes \ ^*V_i \otimes V_i \otimes \cL_{-i-1} \xto{m \otimes \mu} W_i \otimes \cL_{-i}  $
    where $m$ is the adjoint of $\delta$ and $\eta$ is the unit morphism.
\end{enumerate}
\end{lemma}
\begin{remark}
We will think of the morphism $m \colon W_{i+1} \otimes \ ^*V_i \to W_i$ as a type of module multiplication by elements of $^*V_i$.
\end{remark}
To define a Koszul complex, we need the data of a quadratic $D$-algebra $A$, a left $A$-module $\underline{\cL}$ in {\sf C} and a locally finite right $^!A$-module $W = \oplus W_{ij}$.
\begin{definition}\label{def:KoszulComplex}
We define the {\em $W$-Koszul complex} of $\underline{\cL}$ to be
\begin{equation} \label{eq.WKoszul}
W \otimes \underline{\cL} \colon
\ldots \to
W_{i+1} \otimes \cL_{-i-1} \xto{d}
W_{i} \otimes \cL_{-i} \to
W_{i-1} \otimes \cL_{-i+1} \to \ldots
\end{equation}
where the differential is
$$
W_{i+1} \otimes \cL_{-i-1}
\xto{1 \otimes \eta \otimes 1}
W_{i+1} \otimes \ ^!A_{i+1,i} \otimes A_{i,i+1} \otimes \cL_{-i-1}
\xto{m \otimes \mu}
W_{i} \otimes \cL_{-i}
$$
where $\eta$ is the unit map (noting $^!A_{i+1,i} = \ ^*A_{i,i+1}$) and $m,\mu$ are the $^!A$ and $A$-module multiplication maps.
\end{definition}

To prove the $W$-Koszul complex is actually a complex, we will need the following
\begin{lemma} \label{lemma.KjIsComplex}
\begin{enumerate}
    \item The composite of unit morphisms
    $$ D_{i+1} \xto{\eta} \ ^*A_{i,i+1} \otimes A_{i,i+1} \xto{1 \otimes \eta \otimes 1} \ ^*A_{i,i+1} \otimes \,^*A_{i-1,i} \otimes A_{i-1,i} \otimes A_{i,i+1}$$
    is the unit morphism on $A_{i-1,i} \otimes A_{i,i+1}$.
    \item The following is a complex
    $$
    D_{i+1} \xto{\eta} \ ^!A_{i+1,i} \otimes \ ^!A_{i,i-1} \otimes A_{i-1,i} \otimes A_{i,i+1} \xto{^!\mu \otimes \mu} \ ^!A_{i+1,i-1} \otimes A_{i-1,i+1}
    $$
    where the first map is the composite in part~(1) and $\mu, \,^!\mu$ are the algebra multiplication maps in $A,\,^!A$ respectively.
\end{enumerate}
\end{lemma}
\begin{proof}
Part~(1) is just a simple calculation. For part~(2), let $I \subset A_{i-1,i} \otimes A_{i,i+1}$ be the space of quadratic relations on $e_{i-1}Ae_{i+1}$. Pick a $D_{i-1}$-basis $v_1, \ldots , v_n$ for $A_{i-1,i} \otimes A_{i,i+1}$ such that $v_1,\ldots,v_l$ is a basis for $I$, and let $^*v_1, \ldots, \,^*v_n$ be the dual basis. Then $(\,^!\mu \otimes \mu)(\sum \ ^*v_j \otimes v_j) = 0$ since $v_j \in I$ for $j \leq l$ whilst $^*v_j \in \ ^{\perp}I$ for $j > l$.
\end{proof}

\begin{proposition}  \label{prop.KjIsComplex}
The chain of morphisms $W \otimes \underline{\cL}$ in Equation~\ref{eq.WKoszul} is indeed a complex.
\end{proposition}
\begin{proof}
Since the differential in $W \otimes \underline{\cL}$ is defined using module multiplication which is associative, the result follows from Lemma \ref{lemma.KjIsComplex}.
\end{proof}

Before we give our definition of Koszul algebra, we introduce the following notation.  Let $B$ be a $\mathbb{Z}$-indexed $D$-algebra.  We abuse notation by letting $e_{j}B^{*}$ denote the $D_{j}-B$-bimodule with $$
(e_{j}B^{*})_{i}:= \Hom_{D_{j}}(B_{ij},D_{j}) = B_{ij}^*
$$
and with right multiplication induced in the usual way by the left $B$-module structure of $Be_{j}$.  In constructing the Koszul complex, defined below, the only graded right $^!A$-modules we will use are those of the form $e_{j}{}^{!}A^{*}$.


\begin{definition}  \label{defn.AbstractKoszul}
Given an $A$-module $\underline{\cL}$ in an additive category {\sf C}, the {\em degree $j$  Koszul complex} is the Koszul complex ${e_j}{}^!A^* \otimes \underline{\cL}$.
\end{definition}



\begin{definition}  \label{defn.Koszul}
Suppose $A$ is a quadratic $D$-algebra, ${\sf C}$ is the category of graded right $A$-modules and $\underline{\cL}$ is the module $\cL_{-j} = e_j A$. We say $A$ is {\it Koszul} if, for all $j \in \mathbb{Z}$, the degree $j$ Koszul complex ${e_j}^!A^*\otimes \underline{\cL}$ is a resolution of the corresponding simple module $e_jA/(e_jA)_{>j}$.
\end{definition}

\begin{remark} \label{rem.dualkoszul}
Suppose $k$ is a field, $D_{i}=k$ for all $i \in \mathbb{Z}$ and $A$ is a quadratic $D$-algebra.  If
$$
 \,^!A^{*}_{l,j} \otimes A_{l,q} \overset{\psi}{\longrightarrow} \,^!A^{*}_{l-1,j} \otimes A_{l-1,q}
$$
is the $q$th component of the differential in the degree $j$ Koszul complex, then it is straightforward to show, in light of the fact that $({}^!A)^! \cong A$, that the $k$-vector space dual of $\psi$:
$$
A^{*}_{l-1,q} \otimes \,^!A_{l-1,j} \longrightarrow A^{*}_{l,q} \otimes \,^!A_{l,j}
$$
is the $j$th component of the differential in the degree $q$ Koszul complex of $\,^!A$.  Furthermore, if $A$ is Koszul, then the $q$th degree component of the Koszul complex
$$
 \,^!A^{*}_{l,j} \otimes A_{l,q} \longrightarrow \,^!A^{*}_{l-1,j} \otimes A_{l-1,q} \longrightarrow \,^!A^{*}_{l-2,j} \otimes A_{l-2,q}
$$
is exact if either $j<l-1$ or $j=l-1$ and $q \geq l$.  Dualizing, we get that
$$
 \,A^{*}_{l-2,q} \otimes \,^!A_{l-2,j} \longrightarrow \,A^{*}_{l-1,q} \otimes \,^!A_{l-1,j} \longrightarrow \,A^{*}_{l,q} \otimes \,^!A_{l,j}
$$
is exact under the same conditions.  Thus, in case $q \geq l$, we get an exact sequence
$$
\,A^{*}_{l-2,q} \otimes e_{l-2}\,^!A \longrightarrow \,A^{*}_{l-1,q} \otimes e_{l-1}\,^!A \longrightarrow \,A^{*}_{l,q} \otimes e_{l}\,^!A,
$$
while if $q=l-1$, the complex is exact in degrees $<l-1$.
 Thus, $\,^!A$ is Koszul.

 Note that an analogous result holds if $A$ is a negatively graded quadratic algebra generated by $\{A_{i,i-1}\}_{i \in \mathbb{Z}}$ and we define the Koszul complex in the obvious way.  It follows that if $A$ is a positively graded quadratic $D$-algebra such that ${}^{!}A$ is Koszul, then $A$ is Koszul (see \cite[Proposition 2.9.1]{ginz} for a proof in the $\mathbb{Z}$-graded case).
\end{remark}

\section{Morphisms to endomorphism algebras}  \label{sec.MapToEndo}

In the last section, we saw that given a quadratic $D$-algebra $A$ and an $A$-module $\underline{\cL}$ in some additive category {\sf C}, we may construct Koszul complexes. In this section, we show conversely that a type of complex, called a {\it pre-helix} (which resembles part of a Koszul complex), can be used to construct an $A$-module and hence a morphism from $A$ to the endomorphism algebra $\End(\underline{\cL})$.

Let $A$ be a quadratic $D$-algebra. Below, we will need to use its Koszul complex which in degree $j$ is
\begin{equation} \label{eqn.KoszulCompA}
\cdots \longrightarrow \,^!A^{*}_{j+2,j} \otimes e_{j+2}A \xto{\psi_2} \,^!A^{*}_{j+1,j} \otimes e_{j+1}A \xto{\psi_1} \,^!A^{*}_{j,j}\otimes e_{j}A
\end{equation}
The following is well-known and easily verified.
\begin{proposition}  \label{prop.KoszulEndTerms}
\begin{enumerate}
    \item The morphism $\psi_1$ above is the composite of the natural isomorphism $^!A^{*}_{j+1,j} \simeq A_{j,j+1}$ (tensored with $e_{j+1}A$) and the multiplication map $A_{j.j+1} \otimes e_{j+1}A \to e_jA$.
    \item The morphism $\psi_2$ is the composite of the inclusion of relations
    $\rho \colon ^!A^{*}_{j+2,j} \hookrightarrow \,^!A^{*}_{j+1,j} \otimes \,^!A^*_{j+2,j+1}$ and multiplication
    $$
    A_{j+1.j+2} \otimes e_{j+2}A \to e_{j+1}A,
    $$
    where the isomorphism $^!A^*_{j+2,j+1} \simeq A_{j+1,j+2}$ in (1) above has been used.
\end{enumerate}
\end{proposition}

\begin{definition}  \label{def:prehelix}
A {\it pre-helix over } $A$ in {\sf C} consists of the data of
\begin{enumerate}
\item{} a $D$-module $\underline{\cL}$ in {\sf C} such that the natural morphisms $D_i \to \End \cL_{-i}$ are isomorphisms,
\item{} for each $j \in \bZ$, complexes of the form
\begin{equation} \label{eqn.PreHelix}
^!A^{*}_{j+2,j} \otimes \mathcal{L}_{-j-2}  \overset{\phi_{2}}{\longrightarrow} \,^!A^{*}_{j+1,j} \otimes \mathcal{L}_{-j-1} \overset{\phi_{1}}{\longrightarrow} \,^!A^{*}_{j,j}\otimes \mathcal{L}_{-j}
\end{equation}
\end{enumerate}
such that
\begin{enumerate}[(a)]
\item{} $(\mathcal{L}_{-j-1},\phi_1)$ yields a map
$$
 \mu_j \colon A_{j,j+1} \stackrel{\text{can}}{\simeq} \,^!A^{*}_{j+1,j} \otimes D_{j+1} \xto{(\mathcal{L}_{-j-1},\phi_1)} \,^!A^{*}_{j,j} \otimes (\cL_{-j-1},\cL_{-j})  \stackrel{\text{can}}{\simeq} (\cL_{-j-1},\cL_{-j})
$$
which is left $D_j$-linear (it is automatically right $D_{j+1}$-linear), and

\item{} $(\mathcal{L}_{-j-2},\phi_2)$ yields a map
$$^!A^{*}_{j+2,j} \stackrel{\text{can}}{\simeq}
 \,^!A^{*}_{j+2,j} \otimes D_{j+2} \xto{(\mathcal{L}_{-j-2},\phi_2)} \,^!A^{*}_{j+1,j} \otimes (\cL_{-j-2},\cL_{-j-1})
$$
which factors as the inclusion of relations map
$$
\rho\colon ^!A^*_{j+2,j} \to ^!A^*_{j+1,j} \otimes A_{j+1,j+2}
$$
in Proposition~\ref{prop.KoszulEndTerms}(2) and
$$1 \otimes \mu_{j+1} \colon ^!A^*_{j+1,j} \otimes A_{j+1,j+2} \to \,^!A^*_{j+1,j} \otimes (\cL_{-j-2},\cL_{-j-1}).$$
\end{enumerate}
If the morphisms $\mu_j$ in part~(a) are isomorphisms for all $j$, we say that the pre-helix is {\it complete}.
\begin{remark}
The notion of a pre-helix is meant to be an analogue of the notion of linear systems. In the latter setup, we have a number of global sections of some line bundle $\cL$ on a $k$-scheme $X$ or, more invariantly, a morphism of vector spaces $V \to H^0(X,\cL)$. The maps $\mu_j$ in Definition~\ref{def:prehelix}(a) are analogues of this map so complete pre-helices are analogues of complete linear systems.
\end{remark}
\end{definition}

\begin{proposition} \label{prop.MapToEndo}
If $\underline{\mathcal{L}}$ is a pre-helix over $A$, there exists an induced morphism of connected $\mathbb{Z}$-indexed algebras $A \longrightarrow \End(\underline{\mathcal{L}})$. When $\underline{\cL}$ is complete, this is an isomorphism in degree one.
\end{proposition}
\begin{proof}
We construct an $A$-module structure on $\underline{\cL}$. Let $\underline{V}$ denote the degree one part of $A$ so by definition, $A = T(\underline{V})/ \langle (\,^!A^{*}_{j+2,j})_{j \in \bZ})\rangle$.  The ``multiplication'' maps $\mu_j$ in Definition~\ref{def:prehelix}(a) define an algebra morphism $T(\underline{V}) \to \End(\underline{\cL})$ which we wish to show factors through $A$.

By Lemma \ref{lemma.factorViaUnit}(1) and property (a) in Definition \ref{def:prehelix}, the map $\phi_1$ is the composition
$$
A_{j,j+1} \otimes \mathcal{L}_{-j-1} \overset{\mu_{j} \otimes \mathcal{L}_{-j-1}}{\rightarrow} (\mathcal{L}_{-j-1},\mathcal{L}_{-j}) \otimes \mathcal{L}_{-j-1} \rightarrow \mathcal{L}_{-j},
$$
where the rightmost arrow is canonical.  Thus, $\phi_{1}$ is just $T(\underline{V})$-module multiplication.

Similarly, by Lemma \ref{lemma.factorViaUnit}(1) and property (b) in Definition \ref{def:prehelix}, $\phi_2$ is the composite of maps i) induced by the inclusion of relations $^!A^{*}_{j+2,j} \hookrightarrow \,^!A^{*}_{j+1,j} \otimes \,^!A^*_{j+1,j+2}$ and ii) $T(\underline{V})$-module multiplication $A_{j+1.j+2} \otimes \cL_{-j-2} \to \cL_{-j-1}$. The fact that (\ref{eqn.PreHelix}) is a complex implies $\phi_1\phi_2 = 0$, which thus amounts to saying that the relations $^!A^{*}_{j+2,j}$ act trivially on $\underline{\cL}$ so $\underline{\cL}$ is actually an $A$-module and the proposition is proved.

\end{proof}
\begin{remark}  \label{rem:preIsKoszul}
Under the hypotheses of the proposition, $\underline{\cL}$ is an $A$-module so we may speak of the Koszul complexes ${e_j}^!A^* \otimes \underline{\cL}$. The proof above shows that the complex (\ref{eqn.PreHelix}) is actually the start of this Koszul complex.
\end{remark}

There is an elementary converse to Proposition~\ref{prop.MapToEndo}.
\begin{proposition}  \label{prop:AmodGivesPreHelix}
Let $\underline{\cL}$ be a left module over a quadratic $D$-algebra $A$ such that the natural morphisms $D_i \to \End_{\sf C} \cL_{-i}$ are isomorphisms. Then $\underline{\cL}$ is a pre-helix over $A$ when enriched with the data of the Koszul complexes.
\end{proposition}
\begin{proof}
We need only check that the Koszul complexes satisfy property (b) in Definition~\ref{def:prehelix}. This follows from Lemma~\ref{lemma.factorViaUnit} and Proposition~\ref{prop.KoszulEndTerms}.
\end{proof}

Given a pre-helix $\underline{\cL}$ in an additive category {\sf C}, we defined in \cite{morph}, the quadratic algebra $\mathbb{S}^{nc}(\underline{\mathcal{L}})$ as the quadratic part of $\End({\underline{\mathcal{L}}})$, that is, the unique quadratic algebra coinciding with $\End({\underline{\mathcal{L}}})$ in degrees one and two, and having the same relations in degree 2.

\begin{proposition}
Let $\underline{\cL}$ be a pre-helix over $A$ and let $\psi: A \longrightarrow \End(\underline{\mathcal{L}})$ be the map from Proposition \ref{prop.MapToEndo}. Then $\psi$ lifts uniquely to an algebra morphism $\tilde{\psi} \colon A \longrightarrow \mathbb{S}^{nc}(\underline{\mathcal{L}})$. Furthermore, $\tilde{\psi}$ is an isomorphism if and only if $\underline{\cL}$ is complete and $\psi$ is an injection in degree $2$.
\end{proposition}

\begin{proof}
The morphism $\psi$ lifts to $\bS^{nc}(\cL)$ since the relations in $A$ are all quadratic.  If $\tilde{\psi}$ is an isomorphism, then it is, in particular, in isomorphism in degree one. Hence $\underline{\cL}$ is complete. In addition, since the canonical map $\mathbb{S}^{nc}(\underline{\mathcal{L}}) \longrightarrow \End(\underline{\mathcal{L}})$ is an injection in degree two by definition of $\mathbb{S}^{nc}(\underline{\mathcal{L}})$, $\psi$ is an inclusion in degree two.

Conversely, suppose $\underline{\cL}$ is complete and $\psi$ is an injection in degree two.  Since $A$ is quadratic, it factors through $\mathbb{S}^{nc}(\underline{\mathcal{L}})$ via $\tilde{\psi} \colon A \longrightarrow \mathbb{S}^{nc}(\underline{\mathcal{L}})$.  In particular, $\tilde{\psi}$ is a map of quadratic algebras which is an isomorphism in degree zero, one and two, whence the result.
\end{proof}

\begin{remark}
Suppose $\underline{\cL}$ is a helix in the sense of \cite{morph}. Let $A=\mathbb{S}^{nc}(\operatorname{Hom}(\mathcal{L}_{-1}, \mathcal{L}_{0}))$ which we showed in \cite[Proposition~3.6(2)]{morph} to be isomorphic to $\bS^{nc}(\underline{\cL})$. Furthermore, \cite[Corollary~3.7]{morph} gives a morphism of $\bZ$-indexed algebras $\psi\colon A \to \End(\underline{\cL})$ so $\underline{\cL}$ is an $A$-module. Proposition~\ref{prop:AmodGivesPreHelix} shows that $\underline{\cL}$ is actually a pre-helix over $A$ when furnished with the Koszul complexes.
We studied the map $\psi$ in detail in the case that $A$ is the noncommutative symmetric algebra of a finite dimensional vector space over $k=\mathbb{C}$, ${\sf C}$ is the category of quasi-coherent sheaves over a smooth elliptic curve over $k$ and $\underline{\mathcal{L}}$ is an elliptic helix of period two determined by the line bundles $\cL_{-1}, \cL_0$ which satisfy $\deg \cL_0 > \cL_{-1} + 1$. In fact, we showed \cite[Corollary 5.8]{morph} that ${\sf Proj}\ \End(\underline{\cL})$ is in some precise sense a double cover of the noncommutative projective line ${\sf Proj}\ A$. One of the goals of this paper is to give period 3 analogues of this to embed noncommutative elliptic curves into noncommutative projective planes.
\end{remark}

\section{Helices and geometry} \label{sec.geom}

In this section, we introduce our notion of a helix which generalises the notion of Euler exact sequences on projective spaces $\bP^n$ as well as their pullbacks via morphisms $f \colon X \to \bP^n$.  The notion should be considered a noncommutative analogue of linear systems and, in fact, they will induce noncommutative morphisms to noncommutative projective spaces.


\begin{definition}  \label{def:helix}
Let $A$ denote a quadratic $D$-algebra and let $\underline{\mathcal{L}}$ be a pre-helix over $A$ in an abelian category {\sf C}. We say that
\begin{itemize}
\item{} $\underline{\mathcal{L}}$ is a {\it helix of length $2$ over $A$} if for all $j \in \mathbb{Z}$, the pre-helix complexes given as in (\ref{eqn.PreHelix}) extend to exact sequences
\begin{equation} \label{eqn.littleeuler}
0 {\longrightarrow} \,^!A^{*}_{j+2,j} \otimes \mathcal{L}_{-j-2} \overset{\phi_2}{\longrightarrow}
\,^!A^{*}_{j+1,j} \otimes \mathcal{L}_{-j-1} \overset{\phi_1}{\longrightarrow} \,^!A^{*}_{j,j}\otimes \mathcal{L}_{-j} \longrightarrow 0,
\end{equation}
and
\item{} $\underline{\mathcal{L}}$ is a {\it helix of length $n$ over $A$}, where $n \geq 3$, if for all $j \in \mathbb{Z}$, and for $3 \leq i \leq n$, there exist finite-dimensional right $D_{j+i}$-modules $V_{j+i,j}$ such that the pre-helix complexes given as in (\ref{eqn.PreHelix}) extend to exact sequences as follows:
\begin{equation} \label{eqn.bigeuler}
0 \longrightarrow V_{j+n,j} \otimes \mathcal{L}_{-j-n} \longrightarrow \cdots \longrightarrow V_{j+3,j} \otimes \mathcal{L}_{-j-3} \longrightarrow
\end{equation}
\begin{equation*}
\,^!A^{*}_{j+2,j} \otimes \mathcal{L}_{-j-2} \overset{\phi_2}{\longrightarrow}
\,^!A^{*}_{j+1,j} \otimes \mathcal{L}_{-j-1} \overset{\phi_1}{\longrightarrow} \,^!A^{*}_{j,j}\otimes \mathcal{L}_{-j} \longrightarrow 0.
\end{equation*}
Abusing terminology, we will also refer to these exact sequences as helices. The helix is {\it complete} if the pre-helix is.
\end{itemize}
\end{definition}

\begin{example}  \label{eg:classicalEuler}
Let $V$ be an $n+1$-dimensional vector space over a field $k$ and $S(V)$ be the symmetric algebra and $A$ the $\bZ$-indexed algebra associated to $S(V)$. Its Koszul dual is the indexed algebra associated to the classical Koszul dual $S(V)^! = \bigwedge(V^*)$ which is just the exterior algebra. Then $\cL := (\cO(i))$ is a helix of length $n$ whose exact sequences are the Euler exact sequences $K_j:=\oplus_i \bigwedge^iV \otimes \cO(j-i)$. More generally, if $X$ is a projective scheme and $f \colon X \to \bP^n$ a morphism, then $(f^* \cO(i))$ is a helix of length $n$ in $\Coh(X)$.
\end{example}

\begin{example}  \label{eg:perverseNCmorphism}
Morphisms of noncommutative spaces are usually defined as a pair of adjoint functors mimicking $f^*,f_*$. Unfortunately, this definition does not recover the commutative definition when those spaces are just commutative varieties. To take a simple example, consider a pair of adjoint functors
$$f^* \colon \Coh(\bP^1) \to \Coh(\Spec k), \ f_* \colon \Coh(\Spec k) \to \Coh(\bP^1),$$
and suppose $\cG := f_* k \in \Coh(\bP^1)$.  Then, for given $\cF \in \Coh(\bP^1)$, we have
$$(f^* \cF)^* = \Hom_k(f^*\cF,k) = \Hom_{\bP^1}(\cF, f_* k) = \Ext^1_{\bP^1}(\cG(2),\cF)^*.$$
The commutative morphisms $f \colon \Spec k \to \bP^1$ correspond to $\cG$ being a skyscraper sheaf, but there are many other possibilities. Even if one imposes the condition that $f^*$ preserves structure sheaves, that still leaves the possibility that $\cG = \cO_{\bP^1}$ which is realized by an adjoint pair with $f^* = H^1(\bP^1,(-)(-2))$. However, $f^{*}$ does not preserve exactness of the (twist of the) Euler sequence
$$
0 \to \cO \to \cO(1)^{\oplus 2} \to \cO(2) \to 0
$$
which suggests that a helix type condition is a useful hypothesis to impose on morphisms of noncommutative spaces.
\end{example}

Throughout the remainder of this section, {\it we assume $\underline{\mathcal{L}}$ is a complete helix of length $n$ over a quadratic algebra $A$, with corresponding exact sequences being given by those in Definition \ref{def:helix}}. We wish to show that helices with good homological properties define noncommutative morphisms to noncommutative projective spaces. We need a preliminary result.

\begin{lemma} \label{lemma.lessurj}
Let ${\sf C}$ be an abelian category and suppose
\begin{equation} \label{eqn.les}
\cM_* \colon  0 \longrightarrow \mathcal{M}_{n} \longrightarrow \cdots \longrightarrow \mathcal{M}_{1} \longrightarrow \mathcal{M}_{0} \longrightarrow 0
\end{equation}
is an exact sequence in ${\sf C}$.  Suppose there exists $X \in {\sf C}$ such that $^p(X,\cM_i) = 0$ for all $p>0$ and $i = 2, \ldots, n$. Then the complex $(X,\cM_*)$ is exact.
\end{lemma}
\begin{proof}
Let $\Omega_j < \cM_j$ be the syzygies in $\cM_*$ so that there are exact sequences of the form
\begin{equation}  \label{eq:syzygy}
0 \to \Omega_j \to \cM_j \to \Omega_{j-1} \to 0
\end{equation}
for $j = 1, \ldots , n-1$ and $\Omega_{n-1} = \cM_n$. Hence $^p(X, \Omega_{n-1}) = 0$ for all $p>0$ and downward induction on $j$ using (\ref{eq:syzygy}) shows $^p(X, \Omega_j) = 0$ for all $j$. Then $(X, -)$ preserves exactness of the sequences (\ref{eq:syzygy}) and hence of $\cM_*$.
\end{proof}

Below, given a $\bZ$-indexed algebra $C$, we let $C_{\geq 0}$ be the positively graded subalgebra with $(i,j)$-th component $C_{ij}$ for all $i \leq j$.

\begin{proposition} \label{prop.fg}
Let $\underline{\cL}$ be a complete helix of length $n$ over $A$. Suppose there exists an $m \geq 0$ such that for all $l \geq m$, $^{j}(\mathcal{L}_{i}, \mathcal{L}_{i+l})=0$ for all $j >0$ and for all $i \in \bZ$. Then $B:= \End(\underline{\cL})_{\geq 0}$ satisfies the following:
\begin{enumerate}
\item The start of the Koszul complex for the $A$-module $B$ is exact for all $v \geq m+n$,
$$
^!A_{i+2,i}^* \otimes B_{i+2,i+v} \to A_{i,i+1} \otimes B_{i+1,i+v} \xto{\mu} B_{i,i+v}  \to 0
$$
(recall $\mu$ here is module multiplication),
\item{} for all $i\in \bZ$, the right $B$-module $(e_{i}B)_{> i}$ is generated by
$$
B_{i,i+1}, \ldots, B_{i,i+m+n-1},
$$

\item{} for all $i$ we have the following left $D$-module decomposition
$$Be_i =  AB_{i-m-n+1,i}\oplus B_{i-m-n+2,i} \oplus \ldots \oplus  B_{ii}.$$
\end{enumerate}
\end{proposition}
\begin{proof}
We need only prove part~(1) since parts (2) and (3) then follow from surjectivity of $\mu$ in part~(1). By hypothesis
$$
{}^{p}(\mathcal{L}_{-j-v}, \mathcal{L}_{-j-n})=\cdots={}^{p}(\mathcal{L}_{-j-v},\mathcal{L}_{-j-2})=0
$$
holds for all $p>0$ and $v \geq m+n$. The result thus follows from Lemma~\ref{lemma.lessurj} on applying $(\mathcal{L}_{-j-v},-)$ to (\ref{eqn.littleeuler}) or (\ref{eqn.bigeuler}), and noting that this sends the Koszul complex for $\underline{\cL}$ to the Koszul complex for $B$.
%
\end{proof}

Before we use Proposition \ref{prop.fg} to obtain a generalization of \cite[Theorem 4.2]{morph} (Theorem \ref{theorem.morph}), we recall some terminology from \cite{morph}.  Suppose $C$ is a positively graded, connected $\mathbb{Z}$-indexed algebra.  We let ${\sf Gr }C$ denote the category of graded right $C$-modules.  A graded right $C$-module $M$ is {\it right bounded} if $M_{n} = 0$ for all $n >> 0$.  We let ${\sf Tors }C$ denote the full subcategory of ${\sf Gr }C$ consisting of modules whose elements $m$ have the property that the right $C$-module generated by $m$ is right bounded.  If ${\sf Tors }C$ is a localizing subcategory (or even just a Serre subcategory) of ${\sf Gr }C$, then we may form the quotient ${\sf Gr }C/{\sf Tors }C =: {\sf Proj }C$.  We let $\pi_{C}:{\sf Gr }C \rightarrow {\sf Proj }C$ denote the quotient functor, and sometimes write $\pi$ instead of $\pi_{C}$.

\begin{theorem} \label{theorem.morph}
Let $A$ be a quadratic $D$-algebra and $\underline{\cL}$ be a complete helix of length $n$  over $A$ in an abelian category {\sf C}. Let $B = \End(\underline{\cL})_{\geq 0}$. Suppose furthermore that
\begin{itemize}

\item{} there exists an $m \geq 0$ such that for all $l \geq m$ and all $j>0$ ${}^{j}(\mathcal{L}_{i},\mathcal{L}_{i+l})=0$, and

\item{} for all $i \in \mathbb{Z}$, $B_{ij}$ is finite-dimensional over $D_j = B_{jj}$ for $i+1 \leq j \leq i+n+m-1$.
\end{itemize}
Then
\begin{enumerate}
\item{} $B$ is connected,

\item{} ${\sf Tors }B$ is a localizing subcategory of ${\sf Gr }B$, and

\item{} the restriction functor ${\sf Gr }B \longrightarrow {\sf Gr }A$ induces a functor
$$
{\sf Proj }B \longrightarrow {\sf Proj }A.
$$
\end{enumerate}

\end{theorem}

\begin{proof}
Part~(1) is clear. For part~(2), it suffices by the proof of \cite[Lemma~3.5]{morinyman}, to show that $(e_iB)_{>i}$ is a finitely generated $B$-module for all $i$. This follows by
Proposition~\ref{prop.fg} and our assumptions. Since restriction preserves torsion, part~(3) will follow if we can show ${\sf Tors}\, A$ is localizing. In this case, we see that for all $i$, $(e_iA)_{>i}$ is finitely generated since $A$ is generated in degree one.
\end{proof}

\begin{corollary}  \label{cor:helixOnB}
Let $\underline{\cL}$ be a complete helix of length $n$ over $A$ satisfying the hypotheses of Theorem~\ref{theorem.morph}. Let $B = \End(\underline{\cL})_{\geq 0}$. Then $(\pi e_{-j} B)_{j \in \bZ}$ is a helix of length $n$ over $A$ in ${\sf Proj}\, B$ whose helix structure is obtained by applying $(\underline{\cL},?)$ to the helices (\ref{eqn.littleeuler}) or (\ref{eqn.bigeuler}).
\end{corollary}
\begin{proof}
Given the Ext vanishing hypotheses on the $\cL_i$, we see from Lemma~\ref{lemma.lessurj} that $(\cL_{-j-v},?)$ is exact on (\ref{eqn.littleeuler}) or (\ref{eqn.bigeuler}) whenever $v \geq m+n$. This gives the desired exact helix sequences in ${\sf Proj}\, B$.
\end{proof}

For the remainder of the section, we will utilize internal tensor and hom functors introduced in \cite[Section 4 and Section 5]{morinyman}.  Recall that if $B$ and $C$ are $\mathbb{Z}$-indexed algebras, $M$ is a graded $B$-module and $N$ is a bigraded $B-C$-bimodule, we define a graded right $C$-module
$$
M \intotimesb N := \operatorname{cok} \bigl(\bigoplus_{l,m} M_l \otimes_{B_{l,l}} B_{l,m} \otimes_{B_{m,m}} e_{m}N \xrightarrow{\mu \otimes 1 - 1 \otimes \mu} \bigoplus_n M_n \otimes_{B_{n,n}} e_{n}N \bigr).
$$
We denote the $i$th left-derived functor of $-\intotimesb N$ by $\uTor_{i}^{B}(-,N)$.

If, furthermore, $P$ is an object in ${\sf Gr }C$, we let
$$
\sHomc(e_{i}N,P)
$$
denote the right $B_{ii}$-module with underlying set  $\operatorname{Hom}_{C}(e_{i}N,P)$ and with $B_{ii}$-action induced by the left action of $B_{ii}$ on $e_{i}N$, and we let
$$
{\usHomc}(N,P)
$$
denote the object in ${\sf Gr }B$ with $i$th component $\sHomc(e_{i}N,P)$ and with multiplication induced by left-multiplication of $B$ on $N$.

\begin{lemma} \label{lemma.rest}
Let $\psi:B \rightarrow C$ be a morphism of $\mathbb{Z}$-algebras.  Then the restriction of scalars functor
$$
\psi_{*}:{\sf Gr }C \longrightarrow {\sf Gr }B
$$
has a left-adjoint, denoted by $\psi^{*}$.
\end{lemma}

\begin{proof}
It suffices to exhibit $\psi_{*}$ as right-adjoint to $-\underline{\otimes}_{B} C$.  By \cite[Proposition 5.3]{morinyman}, this follows from the fact that there is an isomorphism of functors $\psi_{*}(-) \cong \usHomc({}_{B}C_{C},-)$.
\end{proof}

We once again remind the reader that for the remainder of this section, we assume $\underline{\mathcal{L}}$ is a complete helix of length $n$ over a quadratic algebra $A$.
 We will use the following elementary fact repeatedly in the proof of our next theorem.

\begin{lemma}  \label{lem:tensorSi}
Let $S_i = e_iA/(e_iA)_{>i}$ and $S_i^{op} = Ae_i/(Ae_i)_{<i}$.
\begin{enumerate}
    \item If $N$ is a left $A$-module generated in degrees $>i$ then $S_i \underline{\otimes}_A N = 0$.
    \item If $M$ is a right $A$-module generated in degrees $<i$ then $M \underline{\otimes}_A S_i^{op} = 0$.
\end{enumerate}
\end{lemma}

\begin{theorem} \label{thm:adjoint}
Let $\underline{\cL}$ be a complete helix of length $n$ over $A$, satisfying the hypotheses of Theorem \ref{theorem.morph}.  Let $\psi: A \rightarrow \End({\underline{\mathcal{L}}})_{\geq 0}=:B$ be the map constructed in Proposition \ref{prop.MapToEndo}.  Then the functor $\psi^{*}$ constructed in Lemma \ref{lemma.rest} descends to a functor
$$
{\sf Proj }A \longrightarrow {\sf Proj }B
$$
which is left-adjoint to the functor in Theorem \ref{theorem.morph}(3).
\end{theorem}
\begin{proof}
By \cite[Lemma 1.1]{paulmorph}, it suffices to prove
\begin{enumerate}
\item{} $-\underline{\otimes}_{A}B$ takes torsion $A$-modules to torsion $B$-modules, and

\item{} the first left-derived functor of $\pi_{B} \psi^{*}$ vanishes on ${\sf Tors }A$.
\end{enumerate}
To prove the first assertion, it suffices to prove that $S_{i} \underline{\otimes}_{A} B$ is torsion, where $S_{i}=e_{i}A/(e_{i}A)_{> i}$.  By Proposition \ref{prop.fg}(3), $Be_j$ is generated as an $A$-module in degrees $> j-m-n$. Hence $S_i \underline{\otimes}_A Be_j = 0$ as soon as $j \geq i + m + n$ and part~(1) is proved.

We now prove part~(2). Note that the first left-derived functor of $\pi_B h^*$ is
$$
\pi_B \underline{\cT or}^A_1(-,B)
$$
so it suffices to show that $\underline{\cT or}^A_1(S_i,B)$ is torsion. From Proposition~\ref{prop.fg}(3), there exists an exact sequence of $A$-modules of the form
$$
0 \to AB_{j-m-n+1,j} \to Be_j \to T \to 0
$$
where $T$ lives in degrees $j-m-n+2, \ldots, j$.

We prove separately that $\underline{\cT or}^A_1(S_i,AB_{j-m-n+1,j})=0 $  and
$\underline{\cT or}^A_1(S_i,T)=0$ for $j$ large enough. Consider the partial projective resolution
$$
0 \to A_{i, i+1} A \to e_i A \to S_i \to 0.
$$
Lemma~\ref{lem:tensorSi}(2) shows that $A_{i,i+1}A \underline{\otimes}_A T = 0$ as soon as $i+1 < j-m-n+2$ so $\underline{\cT or}^A_1(S_i,T)=0$ for $j$ large enough.

We compute $\underline{\cT or}^A_1(S_i,AB_{j-m-n+1,j})$ using the partial projective resolution
$$
0 \to K \to A \otimes B_{j-m-n+1,j} \to AB_{j-m-n+1,j} \to 0.
$$
In view of Lemma~\ref{lem:tensorSi} again, the proof of the theorem will be complete if we can establish the following result.
\begin{lemma}  \label{lem:syzygyGen}
The $A$-module $K$ is generated in degrees $\geq j-m-n$.
\end{lemma}
\begin{proof}
Let $s < j-m-n$ and $r \in K_s$. Since $A$ is generated in degree one, we may lift $r$ to some $\tilde{r} \in A_{s,s+1} \otimes \ldots \otimes A_{j-m-n,j-m-n+1} \otimes B_{j-m-n+1,j}$. Consider the composite map
$$
A_{s,s+1} \otimes \ldots \otimes A_{j-m-n,j-m-n+1} \otimes B_{j-m-n+1,j} \xto{1 \otimes\mu'}
A_{s,s+1} \otimes B_{s+1,j} \xto{\mu} B_{s,j}.
$$
Now by Proposition~\ref{prop.fg}(1) we have
$$(1 \otimes \mu')(\tilde{r}) \in \ker \mu = \text{im} (\rho\colon I_{s,s+2} \otimes B_{s+2,j} \to A_{s,s+1} \otimes B_{s+1,j})$$
where $I_{s,s+2}$ are the quadratic relations in $A_{s,s+2}$. We may thus find $\tilde{r}' \in A_{s,s+1} \otimes \ldots \otimes A_{j-m-n,j-m-n+1} \otimes B_{j-m-n+1,j}$ such that $(1 \otimes \mu')(\tilde{r}) = (1 \otimes \mu')(\tilde{r}')$ and such that $\tilde{r}'$ maps to zero in $A \otimes B_{j-m-n+1,j}$. We may thus replace $\tilde{r}$ with
$$
\tilde{r} - \tilde{r}' \in \ker (1 \otimes \mu') = A_{s,s+1} \otimes \ker \mu'.
$$
But this shows that $K_s = A_{s,s+1}K_{s+1}$ and we are done by induction.
\end{proof}
This completes the proof of Theorem~\ref{thm:adjoint}.
\end{proof}

\section{Relation to elliptic helices of period 3}  \label{sec:ellitpic}

For the remainder of this paper, we will assume that $k$ is an algebraically closed field of characteristic zero which will be our base field so $D = k$.
In \cite{bp}, Bondal-Polishchuk  introduced a different notion of a helix called an elliptic helix of period 3. We show how,  when these live in a $k$-linear abelian category {\sf C}, they give examples of helices in our sense. Since $D=k$, left and right duals coincide, so we will revert to the more traditional notation $A^!$ for the Koszul dual of $A$ rather than $^!A$.

\begin{definition}  \label{def:ellipticExc}
(\cite[Section~7, p.249]{bp})
An object $\cL \in {\sf C}$ is {\it elliptically exceptional} if i) $^j(\cL,\cL) \cong k$ for $j = 0,1$ and is zero otherwise and ii) for any $\cF \in {\sf C}$, the natural pairing $(\cL,\cF) \otimes \,^1(\cF,\cL) \to \,^1(\cL,\cL) = k$ is non-degenerate.
\end{definition}

Unlike in \cite{bp}, we will work in the $k$-linear abelian category {\sf C} as opposed to the derived category, so we introduce the following
\begin{definition}  \label{def:mutable}
An ordered pair $\cE,\cF$ of objects in {\sf C} is said to be {\it left mutable} if $^j(\cE,\cF) = 0$ for $j \neq 0$ and furthermore the evaluation map
$$
\eta \colon (\cE,\cF) \otimes \cE \to \cF
$$
is surjective. In this case we define the {\it left mutation} $L_\cE \cF := \ker \eta$. We also say {\it $\cF$ left mutates through $\cE$ in {\sf C}}. The right handed versions are defined similarly.
\end{definition}

Our abelian category version of Bondal-Polishchuk's elliptic helices is given by the following.
\begin{definition}  \label{def:ellipticHelix}
(\cite[Section~7, p. 250]{bp}) Let $\underline{\cL} = (\cL_i)_{i \in \mathbb{Z}}$ be a sequence of elliptically exceptional objects in {\sf C}. We say $\underline{\cL}$ is an {\it elliptic helix of period three} if
\begin{enumerate}
    \item for all $i<j$ we have $^l(\cL_i,\cL_j) = 0$ for $l \neq 0$ while all the $(\cL_i,\cL_j)$ are finite dimensional and,
    \item for all $i$, $\cL_i$ left mutates through $\cL_{i-1}$, $L^1 \cL_i:= L_{\cL_{i-1}} \cL_i$ left mutates through $\cL_{i-2}$ and $L^2 \cL_i := L_{\cL_{i-2}} L^1 \cL_i$ is isomorphic to $\cL_{i-3}$.
\end{enumerate}
\end{definition}

We need some facts about the endomorphism algebra of elliptic helices proved by Bondal-Polishchuk, though not explicitly stated.  Only the last statement below is new and is an analogue of \cite[Theorem~6.6(1)]{atv1}.  The concept of {\it Frobenious algebra of index $n$} invoked in the following result is defined in \cite[p. 239]{bp}.

\begin{theorem}[\cite{bp}]  \label{thm:BPBisKoszul}
Let $\underline{\cL}$ be an elliptic helix of period 3. Let $A = \bS^{nc}(\underline{\cL})$ be the quadratic part of the endomorphism algebra $B:=\End(\underline{\cL})$. Then the following hold.
\begin{enumerate}
    \item{} The algebra $A$ is $3$-periodic, Koszul, has global dimension three, and is AS-regular of dimension three and Gorenstein parameter three (in the sense of \cite[Definition 7.1]{morinyman}).
    \item{} The canonical map $A \longrightarrow B$ is surjective so in particular, $B$ is generated in degree one.
    \item{} $A^{!}$ is a Frobenius $\mathbb{Z}$-indexed algebra of index 3.
    \item $B$ is graded torsion-free in the sense that as a left and right $B$-module, it has no non-zero finite dimensional submodules.
\end{enumerate}
\end{theorem}
\begin{proof}
Let $B(S)$ be the index 3 Frobenius $\mathbb{Z}$-indexed algebra defined in \cite[p. 251]{bp}. Then, as stated in the proof of \cite[Theorem 7.4, p.~253]{bp}, there is a surjection $B(S)^{!} \rightarrow \End(\underline{\mathcal{L}})$ which is an isomorphism in degree one. Since $B(S)^!$ is quadratic, it induces a surjective homomorphism $B(S)^! \to A$. One deduces readily that this is an isomorphism in degree two from \cite[Equation (7.2)]{bp} in light of the form of $B(S)_{i,i+2}$ given immediately after \cite[Equation (7.2)]{bp}. This, together with \cite[Proposition 4.1]{bp}, Lemma \ref{lemma.Periodic}, Remark \ref{rem.dualkoszul} and \cite[Theorem 7.4]{bp}, yields parts~(1)-(3).

To prove part~(4), it suffices, in view of part~(2) to show that for all $i \in \bZ$, the left and right annihilators of $B_{i,i+1}$ in $B e_i$ and $e_{i+1}B$ respectively are zero. Suppose that $s \in B_{ji} = (\cL_{-i},\cL_{-j})$ annihilates $B_{i,i+1}$. This means that $\ker s \subseteq \cL_{-i}$ contains the image of the evaluation map $\text{ev} \colon B_{i,i+1} \otimes \cL_{-i-1} \to \cL_{-i}$. However, by definition of an ellitpic helix of period 3, $\cL_{-i}$ left mutates through $\cL_{-i-1}$ so ev is surjective. This shows that $\ker s = \cL_{-i}$ so $s=0$ proving the right annilator of $B_{i,i+1}$ is zero. A similar argument gives the left handed statement.
\end{proof}

Now let $\underline{\cL}$ be an elliptic helix of period 3. We wish to show it has the structure of a length 3 helix as per Definition~\ref{def:helix}. Consider first the defining exact sequence
\begin{equation}  \label{eq:rightMutate}
    0 \rightarrow \mathcal{L}_{-j-3} \rightarrow {}^{*}(\mathcal{L}_{-j-3}, \mathcal{L}_{-j-2}) \otimes \mathcal{L}_{-j-2} \xrightarrow{\pi} R_{\cL_{-j-2}}\mathcal{L}_{-j-3} \to 0.
\end{equation}
Furthermore, by \cite[Proposition~7.1]{bp}, we know that this is the same exact sequence as obtained by mutating $R_{\cL_{-j-2}}\mathcal{L}_{-j-3}$ left through $\cL_{-j-2}$. We also consider the defining exact sequence
\begin{equation}  \label{eq:leftMutate}
    0 \to L_{\mathcal{L}_{-j-1}}\cL_{-j}  \rightarrow (\mathcal{L}_{-j-1},\mathcal{L}_{-j})\otimes \mathcal{L}_{-j-1} \rightarrow \mathcal{L}_{-j} \rightarrow 0
\end{equation}
which again, is also the exact sequence of the corresponding right mutation. Now $\cL_{-j-3} \simeq L^2 \cL_{-j}$ means precisely that $R_{\cL_{-j-2}}\mathcal{L}_{-j-3} \simeq L_{\mathcal{L}_{-j-1}}\cL_{-j}$ so we may splice together (\ref{eq:rightMutate}) and (\ref{eq:leftMutate}). This gives the following

\begin{definition}
The {\it canonical helix structure} on the ellitpic helix $\underline{\cL}$ of period three is the one defined by splicing together the above exact sequences to get
\begin{equation} \label{eq:ellipticHelixStructure}
0 \rightarrow \mathcal{L}_{-j-3} \rightarrow {}^{*}(\mathcal{L}_{-j-3}, \mathcal{L}_{-j-2}) \otimes \mathcal{L}_{-j-2} \overset{\psi_{2}}{\rightarrow} (\mathcal{L}_{-j-1},\mathcal{L}_{-j})\otimes \mathcal{L}_{-j-1} \overset{\psi_1}{\rightarrow} \mathcal{L}_{-j} \rightarrow 0.
\end{equation}
\end{definition}
\begin{proposition} \label{prop:canonicalHelixOfElliptic}
The canonical helix structure on an elliptic helix $\underline{\cL}$ of period three is indeed a complete helix over $A:=\bS^{nc}(\underline{\cL})$. In fact, the helix exact sequences (\ref{eq:ellipticHelixStructure}) are Koszul complexes.
\end{proposition}
\begin{proof}
We first show that (\ref{eq:ellipticHelixStructure}) define a pre-helix structure over $A$ so by Remark~\ref{rem:preIsKoszul}, we know that the first few terms coincide with the Koszul complex.

We check properties (a), (b) in Definition~\ref{def:prehelix}. Note that $\psi_1$ in (\ref{eq:ellipticHelixStructure}) is the evaluation homomorphism so $(\cL_{-j-1},\psi_1)$ is the identity and property (a) is verified.

We turn now to proving property (b) in Definition~\ref{def:prehelix}.  To this end, recall that if $I_j$ denotes quadratic relations in $A_{j,j+1} \otimes A_{j+1,j+2}$, then the dual of the exact sequence coming from the inclusion of relations into $A_{j,j+1} \otimes A_{j+1,j+2}$ identifies $A_{2+j,j}^{!}$ with $I_{j}^{*}$, so that $I_{j} \cong A_{2+j,j}^{!*}$.
\begin{lemma} \label{lem:relnInSncViaHelix}
There is a natural isomorphism
$$
(\mathcal{L}_{-j-2},R_{\mathcal{L}_{-j-2}}\mathcal{L}_{-j-3}) \cong I_{j}.
$$
\end{lemma}
\begin{proof}
By definition of $I_{j}$, it suffices to show that $(\mathcal{L}_{-j-2},R_{\mathcal{L}_{-j-2}}\mathcal{L}_{-j-3})$ is the kernel of multiplication
$$
(\mathcal{L}_{-j-1}, \mathcal{L}_{-j}) \otimes (\mathcal{L}_{-j-2}, \mathcal{L}_{-j-1}) \longrightarrow (\mathcal{L}_{-j-2}, \mathcal{L}_{-j}).
$$
To prove this, we note that since  $\underline{\mathcal{L}}$ is an elliptic helix of period 3, the exact sequence (\ref{eq:leftMutate}) can be re-written as
\begin{equation} \label{eqn.rcan}
0 \longrightarrow R_{\mathcal{L}_{-j-2}}\mathcal{L}_{-j-3} \xrightarrow{\iota} (\mathcal{L}_{-j-1}, \mathcal{L}_{-j}) \otimes \mathcal{L}_{-j-1} \longrightarrow \mathcal{L}_{-j} \longrightarrow 0.
\end{equation}

Applying $(\mathcal{L}_{-j-2},-)$ to this yields an exact sequence
\begin{equation} \label{eq:kerMult}
    0 \longrightarrow (\mathcal{L}_{-j-2},R_{\mathcal{L}_{-j-2}}\mathcal{L}_{-j-3}) \xrightarrow{\phi} (\mathcal{L}_{-j-1}, \mathcal{L}_{-j}) \otimes (\mathcal{L}_{-j-2}, \mathcal{L}_{-j-1}) \longrightarrow (\mathcal{L}_{-j-2}, \mathcal{L}_{-j})
\end{equation}
where the rightmost map is composition. The lemma follows by uniqueness of kernels.
\end{proof}

Property~(b) of Definition~\ref{def:prehelix} amounts to showing that $(\cL_{-j-2},\psi_2)$ coincides with $\phi$ above in (\ref{eq:kerMult}) on identifying
$$
^*(\cL_{-j-3},\cL_{-j-2}) = (\mathcal{L}_{-j-2},R_{\mathcal{L}_{-j-2}}\mathcal{L}_{-j-3}).
$$
To see this, note that by definition, $\psi_2$ is the composite of the non-trivial surjection $\pi$ in the exact sequence (\ref{eq:rightMutate}) with $\iota$ in (\ref{eqn.rcan}). Now $(\cL_{-j-2}, \pi)$ is just the identification $^*(\cL_{-j-3},\cL_{-j-2}) \xrightarrow{\simeq} (\mathcal{L}_{-j-2},R_{\mathcal{L}_{-j-2}}\mathcal{L}_{-j-3})$ whilst $(\cL_{-j-2}, \iota)$ is $\phi$ on the nose. This completes the proof that the canonical helix structure defines a pre-helix structure over $A$. Furthermore, by definition of $A$ we have $A_{j,j+1} = (\cL_{-j-1},\cL_{-j})$ so the pre-helix is complete.

Since $A^!$ is Frobenius of index 3, it remains only to check that the injection $\iota_{\text{can}} \colon \cL_{-j-3} \hookrightarrow {}^{*}(\mathcal{L}_{-j-3}, \mathcal{L}_{-j-2}) \otimes \mathcal{L}_{-j-2}$ in (\ref{eq:ellipticHelixStructure}) coincides with the remaining non-trivial map in the Koszul complex, namely,
$$
\iota_{\text{Kos}} \colon A^{!*}_{j+3,j} \otimes \cL_{-j-3} \to {}^{*}(\mathcal{L}_{-j-3}, \mathcal{L}_{-j-2}) \otimes \mathcal{L}_{-j-2}.
$$
Now $A^!$ being Frobenius of index 3 means in particular that $A^!_{j+3,j} \cong k$ so the objects in the Koszul complex and canonical exact sequence (\ref{eq:ellipticHelixStructure}) coincide. Furthermore, since the latter is exact, we know that $\iota_{\text{Kos}}$ maps $\cL_{-j-3}$ into $\iota_{\text{can}}(\cL_{-j-3})$. To see $\iota_{\text{Kos}} \colon \cL_{-j-3} \to \iota_{\text{can}}(\cL_{-j-3})$ is an isomorphism, we need only prove that $\iota_{\text{Kos}}$ is non-zero, for $\End \cL_{-j-3} = k$. To this end, note that $(\cL_{-j-3}, \iota_{\text{Kos}})$ is just the following part of the Koszul complex for $A$,
$$
A^{!*}_{j+3,j} \to A^{!*}_{j+2,j} \otimes A_{j+2,j+3}
$$
which is non-zero since $A$ is Koszul. This completes the proof of the proposition.

\end{proof}

\section{Mutating Vector Bundles on Elliptic Curves}

The key to constructing elliptic helices is to have a good criterion for left and right mutability. In this section, we establish such criteria in the case where ${\sf C} = \Coh(X)$ and $X$ is a smooth elliptic curve over an algebraically closed field $k$ of characteristic zero.

We recall that a simple sheaf $\cF$ on $X$ is one such that $\End \cF = k$ so is either a skyscraper sheaf, or a bundle of coprime rank and degree by Atiyah's classification of indecomposable bundles on $X$ \cite{Ati}. In what follows, we will routinely use the fact that if $\cE$ is an indecomposable bundle of degree $d$ and $d>0$, then $h^0(\cE)=d$, while if $d<0$ then $h^0(\cE)=0$ \cite[Lemma 2]{rudakov2}.

\begin{lemma}  \label{lem:dimHoms}
Let $\cE,\cF$ be simple bundles with $\mu(\cF) > \mu(\cE)$. Then $(\cF,\cE) = 0$ and
$$
\dim (\cE,\cF) = \deg (\cE^* \otimes \cF) = \rank(\cF)\rank(\cE) \left( \mu(\cF) - \mu(\cE)\right)
$$
\end{lemma}
\begin{proof}
Note that $(\cE,\cF) = H^0(\cE^* \otimes \cF)$. Now \cite[Lemma~33]{Ati} shows that
\begin{equation}  \label{eq:dimHoms}
\cE^* \otimes \cF \simeq \cL \otimes \cV
\end{equation}
where $\cV$ is an indecomposable bundle and $\cL$ is a direct sum of line bundles $\cL_i$. We take sheaves of endomorphism rings of both sides. The left hand side gives a direct sum of torsion line bundles by \cite[Lemma~22]{Ati} whilst the right hand side is a tensor product of $\End \cL = \oplus_{i,j} \cL_i^* \otimes \cL_j$ and $\End \cV$. However, as proved in \cite[Lemma~23]{Ati}, $\End \cV$ is the tensor product of a direct sum of torsion line bundles, and a direct sum of self-extensions of $\cO_X$. It follows that all the $\cL_i$ have the same degree, so the same is true of all indecomposable summands of $\cE^* \otimes \cF$. These must all be positive since the degree of $\cE^* \otimes \cF$ is $\rank(\cF)\rank(\cE) (\mu(\cF) - \mu(\cE))>0$. The lemma now follows from Serre duality and the fact \cite[Lemma~2(3)]{rudakov2} that, any indecomposable bundle $\cL_i \otimes \cV$ with positive degree $d$ has $\dim H^0(\cL_i \otimes \cV) = d$.
\end{proof}

The following well-known result is proven in \cite[Appendix A]{tuelliptic}:
\begin{proposition} \label{prop:stableBundles}
Every indecomposable vector bundle on $X$ is semistable; it is stable if and only if it is simple.
\end{proposition}



We introduce some alternative terminology for left mutability which seems more appropriate in our present context.
\begin{definition}  \label{def:Egen}
Let $\cE \in \Coh(X)$. A coherent sheaf $\cF$ on $X$ is said to be {\em generated by $\cE$} or {\em $\cE$-generated} if there exists a surjection of the form $\cE^{\oplus m} \twoheadrightarrow \cF$ or equivalently, the canonical evaluation morphism $(\cE,\cF) \otimes \cE \to \cF$ is surjective.
\end{definition}

We consider the question: given a stable bundle $\cE$, when is a stable bundle $\cF$ generated by $\cE$. The most obvious necessary condition is $(\cE,\cF) \neq 0$ and then that $\dim_k(\cE,\cF) \rank(\cE) > \rank(\cF)$. Hence we fix a stable bundle $\cE$ and simple sheaf $\cF$ with $(\cE,\cF) \neq 0$.

\begin{proposition}  \label{prop:maxEgen}
Consider a stable bundle $\cE$ and bundle $\cF$ with $(\cE,\cF) \neq 0$. There exists a stable subsheaf $\cG<\cF$ which is generated by $\cE$ and satisfies the following properties.
\begin{enumerate}
    \item If $\cF'$ is any $\cE$-generated subsheaf of $\cF$, then $\mu(\cF') \leq \mu(\cG)$.
    \item The evaluation morphism $(\cG,\cF) \otimes \cG \hookrightarrow \cF$ is injective.
    \item If $\cG$ is any $\cE$-generated subsheaf of $\cF$ satisfying the maximal slope property in (1), then either $\mu(\cG) > \mu(\cE)$ or $\cG = \cE$.
\end{enumerate}
\end{proposition}
\begin{definition} \label{def:maxEgen}
We will call the sheaf $\cG$ in the proposition, a {\em maximal slope $\cE$-generated stable subsheaf of $\cF$}
\end{definition}
\begin{proof}
Let $\Gamma = \textup{im}((\cE,\cF) \otimes \cE \to \cF)$. We can decompose $\Gamma = \oplus_{i=1}^m \Gamma_i$ where the $\Gamma_i$ are indecomposable bundles and hence, by Atiyah's classification \cite{Ati} of indecomposable bundles on elliptic curves and Proposition~\ref{prop:stableBundles}, semistable with say non-decreasing slope. We also know that $\Gamma_m$ is the iterated self-extension of a stable subsheaf, say $\cG$. Furthermore, since $\cG$ is isomorphic to a quotient of $\Gamma$, it is also $\cE$-generated.

We prove part (1) and consider an $\cE$-generated subsheaf $\cF'<\cF$. The commutative diagram
\[\begin{CD}
(\cE,\cF') \otimes \cE @>>> \cF'\\
@VVV @VVV \\
(\cE,\cF) \otimes \cE @>>> \cF
\end{CD}\]
shows that $\cF'$ is a subsheaf of $\Gamma$, and since all the indecomposable summands of $\Gamma$ have slope bounded above by $\mu(\cG)$, we see that $\mu(\cF') \leq \mu(\cG)$.

We now prove part (2). Consider the commutative diagram
\[
\begin{CD}
(\cG,\cF)\otimes (\cE,\cG) \otimes E @>{\varepsilon_\cG}>> (\cG,\cF) \otimes \cG \\
@V{\phi}VV @VVV \\
(\cE,\cF) \otimes \cE @>>> \cF
\end{CD}
\]
where $\phi$ is induced by composition of sheaf homomorphisms and the others come from evaluation. Surjectivity of $\varepsilon_\cG$ in the commutative diagram above shows that every homomorphism $\cG \to \cF$ factors through $\Gamma < \cF$, that is $(\cG,\cF) = (\cG, \Gamma)$. Now $\cG$ has maximal slope amongst all the components $\Gamma_i$. Let $I \subseteq \{1,\ldots,m\}$ be the indices where $\Gamma_i$ are iterated self-extensions of $\cG$. Then $(\cG,\cF) = \oplus_{i \in I} (\cG,\Gamma_i) = k^{|I|}$ so the evalution map $(\cG,\cF) \otimes \cG \to \cF$ just identifies $(\cG,\cF) \otimes \cG$ with the direct sum of the copy of $\cG$ in each $\Gamma_i, i \in I$.

It remains to prove (3). Since $\cG$ is $\cE$-generated, $(\cE,\cG) \neq 0$ and $\mu(\cE) \leq \mu(\cG)$. Let $\phi \colon \cE \to \cG$ be a non-zero morphism. If $\mu(\cE) = \mu(\cG)$, then stability ensures that $\phi$ is an isomorphism.
\end{proof}

We have the following dichotomy regarding the evaluation morphism $\varepsilon\colon (\cE,\cF) \otimes \cE \to \cF$.

\begin{theorem} \label{thm:evalDichotomy}
Let $\cE,\cF$ be stable bundles with $\mu(\cE) < \mu(\cF)$. Then exactly one of the following occurs. Either
\begin{enumerate}
    \item $\dim (\cE,\cF) \rank(\cE) \leq \rank(\cF)$ in which case $(\cE,\cF) \otimes \cE \hookrightarrow \cF$ is injective or,
    \item $\dim (\cE,\cF) \rank(\cE) > \rank(\cF)$ in which case $(\cE,\cF) \otimes \cE \twoheadrightarrow \cF$ is surjective.
\end{enumerate}
\end{theorem}
\begin{proof}
It suffices to show that the evaluation morphism $\varepsilon$ is either injective or surjective. Let $\cF_1<\cF$ be a maximal slope $\cE$-generated stable subsheaf of $\cF$. For later use, we also let $\bar{\cF}_1 = \cF_1$.  Proposition~\ref{prop:maxEgen}(2),(3) show that either a) $\mu(\cE) = \mu(\cF_1)$ in which case $\varepsilon$ is injective, or b) $\mu(\cF_1) > \mu(\cE)$. We may assume the latter case holds. We are also done if $\varepsilon$ is surjective, so suppose instead it has a non-zero cokernel $\cC$. From \cite[Lemma~3(2)]{rudakov2}, $\cC$ is also a simple sheaf so is either a skyscraper sheaf which is certainly $\cE$-generated, or a simple bundle with $\mu(\cC) > \mu(\cF) > \mu(\cF_1)$. It follows from Lemma~\ref{lem:dimHoms} that $(\cF_1,\cC) \neq 0$ and we may construct $\bar{\cF}_2$, a maximal slope $\cF_1$-generated stable subsheaf of $\cC$. Note $\bar{\cF}_2$ is also $\cE$-generated since $\cF_1$ is. Let $\cF_2$ be the pre-image of $\bar{\cF}_2$ in $\cF$. We continue inductively constructing a filtration
$$
0 < \cF_1 < \cF_2 < \ldots < \cF_s = \cF
$$
whose successive quotients $\bar{\cF}_i$ are all generated by $\cE$ (and even $\cF_1$) and have slope strictly greater than $\mu(\cE)$ (it will be infinite if $\bar{\cF}_i$ is a skyscraper sheaf). In particular, we have vanishing Ext groups $ \ ^1(\cE,\bar{\cF}_i) = 0$ and hence also $\ ^1(\cE,\cF_i) = 0$.

Let $\Gamma = \textup{im}(\varepsilon \colon (\cE,\cF) \otimes \cE \to \cF)$. It suffices to show by induction on $i$ that $\cF_i$ is $\cE$-generated so lies in $\Gamma$. The case $i=1$ follows from the fact that $\cF_1$ is $\cE$-generated. Consider the exact sequence
$$
0 \to \cF_{i-1} \longrightarrow \cF_i \longrightarrow \bar{\cF}_i \to 0.
$$
Applying $\Hom(\cE,?)$ and using the fact that $\,^1(\cE,\cF_{i-1}) = 0$ we see that any homomorphism $\cE \to \bar{\cF}_i$ lifts to $\cF_i$. Now $\bar{\cF}_i$ is $\cE$-generated so we can find a surjection $\bar{\psi} \colon \cE^s \twoheadrightarrow \bar{\cF}_i$ and lift it to $\psi \colon \cE^s \to \cF_i$. Then $\cF_i = \cF_{i-1} + \textup{im}\,\psi$ which is clearly $\cE$-generated by induction, and hence lies in $\Gamma$.
\end{proof}

\section{Construction of Helices on Elliptic Curves}

Fix a smooth elliptic curve $X$ over an algebraically closed field $k$ of characteristic zero. In this section, we study the
\begin{question}
How do you construct an elliptic helix of period 3 in $\Coh(X)$?
\end{question}
By Serre duality, the elliptically exceptional objects in $\Coh(X)$ are the simple bundles and skyscraper sheaves. Since right mutating past a skyscraper sheaf produces 0, an elliptic helix $\underline{\cL}$ of period 3 can only have simple bundles.
\begin{remark}  \label{rem:onlyCheckMutations}
Suppose now that $\underline{\cL}$ is a sequence of simple bundles for which the ``helical'' property~(2) of Definition~\ref{def:ellipticHelix} holds. Since $\cL_i$ left mutates through $\cL_{i-1}$ and both are stable, we must have $\mu(\cL_{i-1}) < \mu(\cL_i)$. Hence, the slopes are strictly increasing and by stability, $^1(\cL_i,\cL_j) = 0$ for $i>j$. Thus property~(1) of Definition~\ref{def:ellipticHelix} follows from axiom~(2) in this case.
\end{remark}

We thus concentrate on the helical property~(2) which implies in particular that $\underline{\cL}$ is completely determined by a ``thread'' $(\cL_0,\cL_1,\cL_2)$ of simple bundles of increasing slope. Conversely, given such a triple, one can try to generate a helix by left and right mutating, though in general, the process may terminate. Actually it will be more convenient to start equivalently, with the triple $(\cL_0,L_{\cL_1}\cL_2, \cL_1)$. We thus make the following

\begin{definition}  \label{def:triad}
A {\it triad} is a triple $T=(\mathcal{A},\mathcal{B}, \mathcal{C})$ of simple bundles on $X$ of increasing slope. If the slopes are $\mu_0, \mu'_1, \mu_1$, then we say $T$ is a {\it $(\mu_0, \mu'_1, \mu_1)$-triad}. We say that $T$ is {\it right mutable} if $\cA$ and $\cB$ both right mutate past $\cC$ in which case, the {\it right mutation of $T$} is the triple
$$
RT := (\mathcal{C}, R_{\cC} \cA,R_{\cC} \cB).
$$
Left mutability and mutations are defined similarly.
\end{definition}
Hence if $\underline{\cL}$ is an elliptic helix of period 3 in $\Coh(X)$, then the triple $(\cL_0,L_{\cL_1}\cL_2, \cL_1)$ is right mutable and its right mutation $(\cL_1,L_{\cL_2}\cL_3, \cL_2)$ is again a triad.

\begin{definition}  \label{def:partialEllipticHelix}
A {\it partial elliptic helix of period 3} is a set of simple bundles $\cL_i, i \in I$ where $I$ is an interval of consecutive integers and such that properties~(1) and (2) of Definition~\ref{def:ellipticHelix} hold whenever they make sense. We also say the partial helix is {\it indexed by $I$}.
\end{definition}
The following illustrates the inductive procedure we will employ.
\begin{remark}  \label{rem:triadInduction}
Suppose that $T = (\cL_0,\cL'_1,\cL_1)$ is a right mutable triad such that $RT = (\cL_1,\cL'_2,\cL_2)$ is again a right mutable triad with $R^2T = (\cL_2,\cL'_3,\cL_3)$. Then $\cL_0,\cL_1,\cL_2,\cL_3$ is a partial elliptic helix of period 3 so long as $\cL_3$ is a bundle, since the simplicity of $\cL_3$ follows from \cite[Lemma 3]{rudakov2}.
\end{remark}

Let $\underline{\cL}$ be a partial elliptic helix of period 3 indexed on $I$. We define $\cL'_i$ to be $L_{\cL_i} \cL_{i+1}$ or $R_{\cL_{i-1}} \cL_{i-2}$, they being isomorphic when both are defined. In this case we consider the following {\it numerical invariants of $\underline{\cL}$}.
\begin{equation} \label{eq:dsandrs}
    d_i := \deg \cL_i,\ r_i := \rank \cL_i,\ d'_i := \deg \cL'_i,\ r'_i := \rank \cL'_i.
\end{equation}
The following result helps to compute these numbers and analyse the mutability condition.  Before we state it, we recall from \cite[Definition 1]{rudakov2} that a pair of objects $(\cC, \cD)$ of $\mbox{Coh }X$ is a {\it simple pair} if both $\cC$ and $\cD$ are simple bundles and ${}^{l}(\cC, \cD)$ is zero for all but one value of $l$.

\begin{lemma} \label{lemma.rightmutate}
Let $\mathcal{A}$ and $\mathcal{B}$ be simple bundles on $X$ with ranks $r_{\cA},r_{\cB}$ and degrees $d_{\cA}, d_{\cB}$. Suppose $\mu(\mathcal{A})<\mu(\mathcal{B})$ and $\operatorname{dim}(\mathcal{A},\mathcal{B})r_{\mathcal{B}}>r_{\mathcal{A}}$.  Then
\begin{enumerate}
\item $\dim (\cA, \cB) = \left|
\begin{smallmatrix}
d_{\cB} & d_{\cA} \\
r_{\cB} & r_{\cA}
\end{smallmatrix}\right|$
\item{} $\cA$ right mutates through $\cB$ and $R_{\cB} \cA$ is a bundle,
\item{} $R_{\cB} \cA$ has rank $\operatorname{dim}(\mathcal{A},\mathcal{B})r_{\mathcal{B}} - r_{\mathcal{A}}$ and degree $\operatorname{dim}(\mathcal{A},\mathcal{B})d_{\mathcal{B}} - d_{\mathcal{A}}$, and
\item{} $(\mathcal{B},R_{\cB} \cA)$ is a simple pair with $\mu(\mathcal{B})<\mu(R_{\cB} \cA)$.
\end{enumerate}
\end{lemma}
\begin{proof}
Part~(1) is just a convenient restatement of Lemma~\ref{lem:dimHoms} whilst part~(3) follows from the previous parts and the fact that rank and degree are additive on exact sequences.

We prove part~(2) now.
Consider the evaluation map
$$
(\mathcal{B}^{*}, \mathcal{A^{*}}) \otimes \mathcal{B}^{*} \overset{\operatorname{ev}}{\longrightarrow} \mathcal{A}^{*}.
$$
Our hypotheses ensure $\mu(\mathcal{A}^{*}) > \mu(\mathcal{B}^{*})$ and  $\operatorname{dim}(\mathcal{B}^{*}, \mathcal{A}^{*})\rank \mathcal{B}^{*}>\rank \mathcal{A}^{*}$.  Thus, by Theorem \ref{thm:evalDichotomy}, the map $\operatorname{ev}$ is an epimorphism.  Since its kernel is torsion-free, it is a vector bundle. Taking duals, we conclude that the coevaluation map
$$
\cA \to \,^*(\cA,\cB) \otimes \cB
$$
is injective with cokernel $R_{\cB}\cA$ a bundle.

Finally, to show $(\mathcal{B}, R_{\cB}\cA)$ is a simple pair, it suffices to show $(\mathcal{A}, \mathcal{B})$ is a simple pair by \cite[Lemma 3]{rudakov2}. This follows from Serre duality which gives $^1(\cA,\cB) = (\cB,\cA)^* = 0$ since $\cA, \cB$ are stable bundles with slopes $\mu(\cB) > \mu(\cA)$. Also, since $(\cB, R_{\cB} \cA) \neq 0$, we must also have $\mu(\cB) < \mu(R_{\cB} \cA)$.
\end{proof}

Lemma~\ref{lemma.rightmutate} immediately gives
\begin{proposition}  \label{prop:drRecursionRelations}
The numerical invariants of a partial elliptic helix of period 3 satisfy the following recursion relations.
\begin{equation} \label{eqn.unprimed}
\begin{pmatrix} d_{i}' \\ r_{i}' \end{pmatrix} = \begin{vmatrix} d_{i-1} & d_{i-2} \\ r_{i-1} & r_{i-2} \end{vmatrix} \begin{pmatrix} d_{i-1} \\ r_{i-1} \end{pmatrix} - \begin{pmatrix} d_{i-2} \\ r_{i-2} \end{pmatrix}
\end{equation}
and
\begin{equation} \label{eqn.primed}
\begin{pmatrix} d_{i} \\ r_{i} \end{pmatrix} = \begin{vmatrix} d_{i-1} & d_{i-1}' \\ r_{i-1} & r_{i-1}' \end{vmatrix} \begin{pmatrix} d_{i-1} \\ r_{i-1} \end{pmatrix} - \begin{pmatrix} d_{i-1}' \\ r_{i-1}' \end{pmatrix}.
\end{equation}
\end{proposition}

\begin{definition}  \label{def:seed}
A {\it seed} is a triple $(\mu_0,\mu'_1,\mu_1)$ of strictly increasing rational numbers. We write fractions in reduced form $\mu_0 = \frac{d_0}{r_0}, \mu'_1 = \frac{d'_1}{r'_1}, \mu_1 = \frac{d_1}{r_1}$ with positive denominators. The {\it numerical invariants generated by the seed} are the integers $d_i,r_i,d'_i,r'_i, \ i \in \bN$ defined using the recursion relations~(\ref{eqn.unprimed}),(\ref{eqn.primed}).
\end{definition}

\begin{lemma} \label{lemma.claim1}
For $n \geq 1$, the numerical invariants generated by a seed satisfy $\begin{vmatrix} d_{n+1} & d_{n} \\ r_{n+1} & r_{n} \end{vmatrix} = \begin{vmatrix} d_{n} & d_{n}' \\ r_{n} & r_{n}' \end{vmatrix}$.
\end{lemma}

\begin{proof}
This follows from properties of the determinant:
$$
\begin{CD}
\begin{vmatrix} d_{n+1} & d_{n} \\ r_{n+1} & r_{n} \end{vmatrix} & = & \begin{vmatrix} \begin{vmatrix} d_{n} & d_{n}' \\ r_{n} & r_{n}' \end{vmatrix} \begin{pmatrix} d_{n} \\ r_{n} \end{pmatrix} - \begin{pmatrix} d_{n}' \\ r_{n}' \end{pmatrix}, \begin{pmatrix} d_{n} \\ r_{n} \end{pmatrix} \end{vmatrix} & = & \begin{vmatrix} d_{n} & d_{n}' \\ r_{n} & r_{n}'\end{vmatrix}.
\end{CD}
$$
\end{proof}

\begin{lemma} \label{lemma.claim2}
For $n \geq 2$, the numerical invariants generated by a seed satisfy $\begin{vmatrix} d_{n+1} & d_{n+1}' \\ r_{n+1} & r_{n+1}' \end{vmatrix}= \begin{vmatrix} d_{n-1} & d_{n-2} \\ r_{n-1} & r_{n-2} \end{vmatrix}$.
\end{lemma}

\begin{proof}
\begin{eqnarray*}
\begin{vmatrix} d_{n+1} & d_{n+1}' \\ r_{n+1} & r_{n+1}' \end{vmatrix} & = & \begin{vmatrix} \begin{pmatrix} d_{n+1} \\ r_{n+1} \end{pmatrix}, \begin{vmatrix} d_{n} & d_{n-1} \\ r_{n} & r_{n-1} \end{vmatrix} \begin{pmatrix} d_{n} \\ r_{n} \end{pmatrix} - \begin{pmatrix} d_{n-1} \\ r_{n-1} \end{pmatrix} \end{vmatrix} \\
& = & \begin{vmatrix} \begin{vmatrix} d_{n} & d_{n'} \\ r_{n} & r_{n}' \end{vmatrix} \begin{pmatrix} d_{n} \\ r_{n} \end{pmatrix} - \begin{pmatrix} d_{n}' \\ r_{n}' \end{pmatrix}, \begin{vmatrix} d_{n} & d_{n-1} \\ r_{n} & r_{n-1} \end{vmatrix} \begin{pmatrix} d_{n} \\ r_{n} \end{pmatrix} - \begin{pmatrix} d_{n-1} \\ r_{n-1} \end{pmatrix} \end{vmatrix} \\
& = & -\begin{vmatrix} d_{n} & d_{n}' \\ r_{n} & r_{n}' \end{vmatrix}\begin{vmatrix} d_{n} & d_{n-1} \\ r_{n} & r_{n-1} \end{vmatrix} - \begin{vmatrix} d_{n} & d_{n-1} \\ r_{n} & r_{n-1} \end{vmatrix} \begin{vmatrix} d_{n}' & d_{n} \\ r_{n}' & r_{n} \end{vmatrix}+\begin{vmatrix} d_{n}' & d_{n-1} \\ r_{n}' & r_{n-1}\end{vmatrix} \\
& = & \begin{vmatrix} d_{n}' & d_{n-1} \\ r_{n}' & r_{n-1} \end{vmatrix} \\
& = & \begin{vmatrix} \begin{vmatrix} d_{n-1} & d_{n-2} \\ r_{n-1} & r_{n-2} \end{vmatrix} \begin{pmatrix} d_{n-1} \\ r_{n-1} \end{pmatrix}-\begin{pmatrix} d_{n-2} \\ r_{n-2} \end{pmatrix}, \begin{pmatrix} d_{n-1} \\ r_{n-1} \end{pmatrix}  \end{vmatrix} \\
& = & \begin{vmatrix} d_{n-1} & d_{n-2} \\ r_{n-1} & r_{n-2} \end{vmatrix}.
\end{eqnarray*}
\end{proof}

The next result follows immediately from Lemma \ref{lemma.claim1} and Lemma \ref{lemma.claim2}.

\begin{corollary} \label{cor.periodicity}
For $n \geq 3$, the numerical invariants generated by a seed satisfy
$\begin{vmatrix} d_{n+1} & d_{n} \\ r_{n+1} & r_{n} \end{vmatrix} = \begin{vmatrix} d_{n-2} & d_{n-3} \\ r_{n-2} & r_{n-3}\end{vmatrix}$ and $\begin{vmatrix} d_{n+1} & d_{n+1}' \\ r_{n+1} & r_{n+1}' \end{vmatrix}= \begin{vmatrix} d_{n-2} & d_{n-2}' \\ r_{n-2} & r_{n-2}' \end{vmatrix}$.
\end{corollary}
This 3-periodicity of numerical invariants can also be nicely interpreted in terms of triads and their mutations using the following
\begin{definition}  \label{def:Homdim}
Let $T = (\cA,\cB,\cC)$ be a triple of objects in $\Coh(X)$. The {\it Hom dimension} of the triple is
$$
\dim\Hom(\cA,\cB,\cC) = (\dim\Hom(\cA,\cB),\dim\Hom(\cA,\cC),\dim\Hom(\cB,\cC))
$$
\end{definition}

\begin{proposition} \label{prop:HomDimMutation}
Let $T$ be a right mutable triad with Hom dimension $(a,b,c)$. Then $\dim\Hom RT = (b,c,a)$.
\end{proposition}
\begin{proof}
This follows from the lemmas above.
\end{proof}
This means that the $2\times 2$-determinants in the recursion relations (\ref{eqn.primed}),(\ref{eqn.unprimed}) can only be one of the 3 initial determinants of the seed.

We can now give a numerical criterion for when a triad can be repeatedly right mutated to produce a partial elliptic helix of period 3, indexed by $\bN$.

\begin{theorem}  \label{thm:numericalCriterionGenHelix}
Let $T=(\mathcal{L}_{0}, \mathcal{L}_{1}^{'}, \mathcal{L}_{1})$ be a $(\mu_0,\mu'_1,\mu_1)$-triad. Let $r_i,r'_i$ be the integers generated from the seed $(\mu_0,\mu'_1,\mu_1)$ as in Definition~\ref{def:seed}. If $ r_n,r'_n >0$ for all $n$, then for $i \geq 1$, $R^{i}T =: (\mathcal{L}_{i}, \mathcal{L}_{i+1}^{'}, \mathcal{L}_{i+1})$ is a well-defined right mutable triad. Furthermore, $\cL_0.\cL_1,\cL_2,\ldots$ is a partial elliptic helix of period 3.
\end{theorem}
\begin{proof}
From Remarks~\ref{rem:onlyCheckMutations} and \ref{rem:triadInduction}, the last assertion that $\cL_0,\cL_1,\cL_2,\ldots$ is a partial elliptic helix of period 3 will follow from the assertion that all the $R^iT$ are well-defined triads. We prove the latter by induction on $i$.

First note that the initial Hom dimension $\dim \Hom T = (a,b,c)$ consists of positive integers $a,b,c >0$ since $T$ is a triad and Hom dimensions can be computed using Lemma~\ref{lem:dimHoms}. Suppose that $R^{i-1}T$ is a well-defined triad. Thus  $\mu(\cL_{i-1}) < \mu(\cL_i)$ and by hypothesis $\dim (\cL_{i-1},\cL_i) r_i - r_{i-1} = r'_{i+1} > 0$. Hence, by
Lemma~\ref{lemma.rightmutate}, $\cL_{i-1}$ right mutates through $\cL_i$ and $\cL'_{i+1}$ is a bundle. The same argument using the fact that $r_i >0$ shows that $\cL'_i$ right mutates through $\cL_i$ and the right mutation is a bundle $\cL_{i+1}$. Thus $R^iT$ is a well-defined triple. It remains only to show that $\mu(\cL'_{i+1}) < \mu(\cL_{i+1})$. It follows from Proposition~\ref{prop:HomDimMutation} that $\dim \Hom R^iT$ also consists of the positive integers $a,b,c$ above  (possibly permuted), so by Lemma~\ref{lem:dimHoms}, we are done.
\end{proof}

The hypotheses of the theorem naturally prompt
\begin{question}  \label{qu:seedsGenHelix}
Which seeds $(\mu_0,\mu'_1,\mu_1)$ generate numerical invariants with all $r'_i,r_i$ positive?
\end{question}
Note that Bondal-Polishchuk \cite[Proposition~7.3]{bp} constructed an elliptic helix of line bundles or period 3. It can be built from any $(0,\tfrac{3}{2},3)$-triad. This has the nice property that the Hom dimension is the constant $(3,3,3)$ so the recursion relations in (\ref{eqn.unprimed}),  (\ref{eqn.primed}) simplify considerably. We show below, the same is true for the seed $(0,\tfrac{d}{2},d)$ where $d$ is an odd integer greater than three.

\begin{lemma} \label{lemma.constant}
Given the seed $(0,\tfrac{d}{2},d)$, the numerical invariants satisfy  $$\begin{vmatrix} d_{n} & d_{n-1} \\ r_{n} & r_{n-1} \end{vmatrix} = \begin{vmatrix} d_{n} & d_{n}' \\ r_{n} & r_{n}' \end{vmatrix} =
\begin{vmatrix} d'_{n} & d_{n-1} \\ r'_{n} & r_{n-1} \end{vmatrix}
=d$$
for all $n \geq 1$,
\end{lemma}
\begin{proof}
By Proposition~\ref{prop:HomDimMutation}, it suffices to check that all three $2 \times 2$-minors of
$$
\begin{pmatrix} d_{1} & d'_1 & d_{0} \\ r_{1} & r'_1 & r_{0} \end{pmatrix} =
\begin{pmatrix} d & d & 0 \\ 1 & 2 & 1 \end{pmatrix}
$$
equal $d$.

\end{proof}

\begin{proposition}  \label{prop:egSeedWhichGen}
If $d>3$ is an odd integer, then the numerical invariants of the seed $(0,d/2,d)$ defined in Definition~\ref{def:seed} satisfy $r_n,r'_n >0$ for all $n \geq 0$.
\end{proposition}
\begin{proof}
We begin by explicitly computing the numerical invariants $r_{n}$, $d_{n}$ and $\mu_{n}:= d_{n}/r_{n}$ for $n \geq 1$.  By Lemma \ref{lemma.constant} and (\ref{eqn.primed}) and (\ref{eqn.unprimed}), both pairs $r_{n}$, $r_{n}'$ and $d_{n}$, $d_{n}'$ satisfy the following recurrence relations: $a_{n}'=da_{n-1}-a_{n-2}$ and $a_{n}=da_{n-1}-a_{n-1}'$.  Hence, after substitution, we have
$$
a_{n+1}=da_{n}-da_{n-1}+a_{n-2}
$$
for $n \geq 2$.  Using standard recursion relation methods (or induction) we compute $r_{n}$ from the initial values $r_{0}=1=r_{1}$ and $r_{2}=d-2$ and $d_{n}$ from the initial values $d_{0}=0$, $d_{1}=d$ and $d_{2}=d^2-d$. This gives
\begin{lemma}  \label{lem:closedFormrd}
Letting $A=\sqrt{(d-3)(d+1)}$, we have, for $n \geq 0$,
$$
r_{n} = 2^{-n-1}\left((d-1-A)^{n}\left(\frac{d-3}{A}+1 \right)+(d-1+A)^{n}\left(1-\frac{d-3}{A}\right) \right),
$$
and
$$
d_{n}=\frac{2^{-n}d}{A}\left( (d-1+A)^{n}-(d-1-A)^{n} \right).
$$
\end{lemma}

\begin{lemma} \label{lemma.abound}
For $d \geq 5$, $0<d-1-A<1$.
\end{lemma}

\begin{proof}
We must show that $A<d-1$ and $d-2<A$.  We prove the first inequality.  The second is similar.  Since $A^2=d^2-2d-3$ and $(d-1)^2=d^2-2d+1$, it follows that $A^2<(d-1)^2$, whence the first inequality.
\end{proof}
For later purposes, we also need the following result.
\begin{lemma}  \label{lem:rightlimitslope}
\begin{equation*} \lim_{n \to \infty}\frac{d_{n}}{r_{n}} = \frac{2d}{A-(d-3)} \notin \bQ
\end{equation*}
\end{lemma}
\begin{proof}
The value of the limit follows readily from Lemmas~\ref{lem:closedFormrd} and \ref{lemma.abound}. Irrationality follows from the fact that $(d-3)(d+1)$ is not a perfect square, being distinct from $(d-2)^2,(d-1)^2$ and $d^2$.
\end{proof}

Next, we bound $\frac{r_{n+1}}{r_{n}}$. We begin with the following
\begin{lemma} \label{lemma.ineq}
For all $n \geq 1$
$$
\left( 1 - ( \tfrac{d-3}{A}) \right)(d-1+A)^{n} \geq \left(1 +(\tfrac{d-3}{A})\right)(d-1-A)^{n}.
$$
\end{lemma}
\begin{proof}
After rearranging, we must establish
$$
\left( \frac{d-1+A}{d-1-A} \right)^{n} \geq \frac{A+d-3}{A-(d-3)}
$$
where we have used Lemma \ref{lemma.abound}.  For $n=1$, we must show
$$
(d-1+A)(A-(d-3)) \geq (d-1-A)(A+(d-3))
$$
which we see holds on substituting $A^2 = (d-3)(d+1)$.  Now suppose the desired inequality holds for some $n \geq 1$.  To prove the induction step, it suffices to show $d-1+A \geq d-1-A$, which is clear.
\end{proof}

\begin{proposition} \label{prop.ratio}
For $n\geq 1$, $\frac{r_{n+1}}{r_{n}} \geq \frac{d-1}{2}$.
\end{proposition}

\begin{proof}
Using our explicit formula for $r_{n}$, we have
$$
\frac{r_{n+1}}{r_{n}} = \frac{1}{2}\left(\frac{(d-1-A)^{n+1}(\frac{d-3}{A}+1) +(d-1+A)^{n+1} (1-\frac{(d-3)}{A})}{(d-1-A)^{n}(\frac{d-3}{A}+1 )+(d-1+A)^{n} (1-\frac{(d-3)}{A})}\right)
$$
$$
=\frac{1}{2}\left(\frac{(d-1-A)(d-1-A)^{n}(\frac{d-3}{A}+1) +(d-1+A)(d-1+A)^{n} (1-\frac{(d-3)}{A})}{(d-1-A)^{n}(\frac{d-3}{A}+1 )+(d-1+A)^{n} (1-\frac{(d-3)}{A})}\right)
$$
$$
=\frac{1}{2}(d-1)+\frac{1}{2}A\left(\frac{-(d-1-A)^{n}(\frac{d-3}{A}+1) +(d-1+A)^{n} (1-\frac{(d-3)}{A})}{(d-1-A)^{n}(\frac{d-3}{A}+1 )+(d-1+A)^{n} (1-\frac{(d-3)}{A})}\right).
$$
By Lemma \ref{lemma.ineq}, the numerator of the second summand is nonnegative.  In addition, by Lemma \ref{lemma.abound}, the denominator of the second summand is positive.  It follows that $\frac{r_{n+1}}{r_{n}} \geq \frac{d-1}{2}$ for $n\geq 1$ as desired.
\end{proof}
We may now complete the proof of Proposition~\ref{prop:egSeedWhichGen}. Now Proposition~\ref{prop.ratio} shows us that $r_n$ is an increasing function of $n$ so must be positive. Furthermore, $r'_n = dr_{n-1} - r_{n-2}$ must also be positive. This completes the proof of the Proposition.
\end{proof}

\begin{theorem}  \label{thm:egTriadGenHelix}
Let $T$ be a $(0,\tfrac{d}{2},d)$-triad for some odd integer $d\geq 5$. We may repeatedly mutate $T$ to the right to generate a partial elliptic helix of period three, $\cL_0,\cL_1,\cL_2, \ldots$. We may also repeatedly mutate it to the left to complete this to an elliptic helix $\underline{\cL} = (\cL_i)_{i \in \bZ}$ of period 3. Furthermore,
$$
\lim_{n \to -\infty}\mu(\cL_n) = d \biggl{(}\frac{A-(d-1)}{A-(d-3)}\biggr{)}
$$
and this limit is negative and irrational.
\end{theorem}
\begin{definition}
\label{def:helixGenByTriad}
We call $\underline{\cL}$ in the theorem the {\it helix  generated by $T$}.
\end{definition}

\begin{proof}
Proposition~\ref{prop:egSeedWhichGen} ensures that the hypotheses of Theorem~\ref{thm:numericalCriterionGenHelix} hold, from which we conclude that we can repeatedly mutate to the right and generate the partial elliptic helix of period three,  $\cL_0,\cL_1,\cL_2, \ldots$.

To mutate left, we consider the $(0,\tfrac{d}{2},d)$-triad
\begin{equation*} \label{eqn.triple}
\cL_1 \otimes T^*:=
( \mathcal{L}_{1} \otimes \mathcal{L}_{1}^{*}, \mathcal{L}_{1} \otimes \mathcal{L}_{1}^{'*}, \mathcal{L}_{1} \otimes \mathcal{L}_{0}^{*})
\end{equation*}
The argument in the previous paragraph means we may repeatedly mutate this to the right, and so the same is true of the triad $T^* = (\cL_1^*,\cL^{'*}_1, \cL_0^*)$ on tensoring by $\cL_1^*$. This gives a partial elliptic helix of period three of the form $\cL_1^*, \cL_0^*, \cL_{-1}^*, \cL_{-2}^*,\ldots $. Now duals of partial elliptic helices of period three are also partial elliptic helices of period three since the dual functor takes evaluation short exact sequences of vector bundles to coevaluation short exact sequences of vector bundles and vice versa. We thus obtain an elliptic helix $\underline{\cL}$ of period 3.

Finally, we compute the limit of the slopes. By Lemma~\ref{lem:rightlimitslope}, the sequence defined by the right mutations of the triple $\cL_1^* \otimes T$ above has limit slope
$$
\frac{2d}{A-(d-3)}
$$
and this is an irrational number.  Thus,
$$
\lim_{n \to \infty} \mu(\cL^*_{-n}) =
-d+\frac{2d}{A-(d-3)} = -d \biggl{(}\frac{A-(d-1)}{A-(d-3)}\biggr{)},
$$
which must be irrational as well, and positive by Lemma \ref{lemma.abound}.  It follows that $\lim_{n \to -\infty}\mu(\cL_n) = d-\frac{2d}{A-(d-3)}$ is irrational and negative.
\end{proof}








\section{Properties of $\End(\underline{\cL})$ in the equigenerated case}  \label{sec:hilbert}

Fix a smooth elliptic curve $X$. We know from Theorem~\ref{thm:egTriadGenHelix}, that $T = (0,\tfrac{d}{2},d)$-triad with $d \geq 5$ an odd integer generates an elliptic helix of period 3. Since $\dim\Hom T = (d,d,d)$ in this case, for all $i \in \bZ$, the space of generators $(\cL_i,\cL_{i+1})$ in the endomorphism algebra $\End(\underline{\cL})$ has dimension $d$ by Proposition~\ref{prop:HomDimMutation}. In this section, we study helices with this latter property, and study the Hilbert series and relations for $\End(\underline{\cL})$. This generalises some of the key results in \cite{atv1} and \cite{bp}.

\begin{definition}  \label{def:equigen}
We say that a $\bZ$-indexed $k$-algebra $C$ is {\it equigenerated by $d$ elements} if it is generated in degree one and $\dim_k C_{i,i+1} = d$ for all $i \in \bZ$.
\end{definition}

\begin{proposition}  \label{prop:EndLEquigen}
Let $\underline{\cL}$ be an elliptic helix of period 3 in $\Coh(X)$. Then $\End(\underline{\cL})$ is equigenerated by $d$ elements iff one of the following equivalent conditions hold:
\begin{enumerate}
    \item $\underline{\cL}$ is generated by a triad $T$ with $\dim\Hom T = (d,d,d)$.
    \item $d = \dim (\cL_i,\cL_{i+1}) = \dim (\cL_{i+1}, \cL_{i+2})=\dim (\cL_{i+2}, \cL_{i+3})$ for some $i \in \bZ$.
\end{enumerate}
In particular, if $\underline{\cL}$ is generated by a $(0,\tfrac{d}{2},d)$-triad, then $\End(\underline{\cL})$ is equigenerated by $d$ elements.
\end{proposition}
\begin{proof}
This follows from 3-periodicity of the dimensions of Hom spaces Corollary~\ref{cor.periodicity} and  Proposition~\ref{prop:HomDimMutation}.
\end{proof}

\begin{proposition}  \label{prop:KoszulResultionEquigen}
Let $\underline{\cL}$ be an elliptic helix of period 3 in $\Coh(X)$. Suppose that $\End(\underline{\cL})$ is equigenerated by $d$ elements. If $A = \bS^{nc}(\underline{\cL})$ is the quadratic part of $B = \End(\underline{\cL})$ the following hold.
\begin{enumerate}
    \item For all $i \in \bZ$, we have
    $$
    \dim A_{i,i+1} = \dim B_{i,i+1} = d
    $$
    and
    $$
    \dim A_{i,i+2} = \dim B_{i,i+2} = d^2-d.
    $$
    \item $A$ is a Koszul algebra whose Koszul resolutions all have the form
    $$
    0 \to e_{j+3}A \to e_{j+2}A^{\oplus d}  \to e_{j+1}A^{\oplus d} \to e_jA \to A_{jj} \to 0.
    $$
\end{enumerate}
\end{proposition}
\begin{proof}
From Theorem~\ref{thm:BPBisKoszul}, we know that $A$ is Koszul, $A^!$ is Frobenius of index 3 and the morphism $A \to B$ is an isomorphism in degrees 0,1 and 2. Hence, we need only verify part (1) for $A$.  Lemma~\ref{lem:relnInSncViaHelix}, shows that the quadratic relations in $A$ are of the form $A^{!*}_{2+j,j} = (\cL_{-j-2},R_{\cL_{-j-2}}\cL_{-j-3})$. This has dimension $d$, since by Proposition~\ref{prop:HomDimMutation} the Hom dimension of a generating triad and all its mutations is $(d,d,d)$. This shows $\dim A_{i,i+2} = d^2 - d$. Furthermore, we now also know $A^!$ is dimension $d$ in degrees 1 and 2, and one dimensional in degrees 0 and 3 so part (2) also follows.
\end{proof}

\begin{definition}  \label{def:relativeHilbertSeries}
Let $A$ be a $\bZ$-indexed $k$-algebra. The {\it $n$-th relative Hilbert series} of $A$ is defined to be
$$
H_{A,n}(t) : = \sum_{i \in \bZ} \dim A_{n,n+i} t^i.
$$
\end{definition}
\begin{proposition}  \label{prop:HilbertSeries}
Let $\underline{\cL}$ be an elliptic helix of period 3 in $\Coh(X)$ such that $B = \End(\underline{\cL})$ is equigenerated by $d$ elements. Let $A = \bS^{nc}(\underline{\cL})$ be its quadratic part. Then for all $n \in \bZ$.
$$
H_{A,n}(t) = \frac{1}{1 - dt + dt^2 - t^3}, \quad
H_{B,n}(t) = \frac{1-t^3}{1 - dt + dt^2 - t^3}
$$
\end{proposition}
\begin{proof}
We first show that $H_{A,n}$ is independent of $n$. To this end, let $h_{n,i} = \dim A_{n,n+i}$ and note that for $i\leq 2$, it is independent of $n$ by Proposition~\ref{prop:KoszulResultionEquigen}(1). The Koszul resolutions in Proposition~\ref{prop:KoszulResultionEquigen} give the following recursive formula for all $n\in \bZ, i\geq 3$.
$$
h_{n,i} - d h_{n+1,i-1} + d h_{n+2,i-2} - h_{n+3,i-3} = 0.
$$
Induction on $i$ using Proposition \ref{prop:KoszulResultionEquigen}(1) shows that all the $h_{n,i}$ are independent of $n$ so the same is true of the relative Hilbert series and we may write $H_A = H_{A,n}$, The additivity of the usual Hilbert series on the Koszul resolutions now gives
$$
(1 - dt + dt^2 -t^3)H_A(t) = 1
$$
which gives the Hilbert series for $A$ stated above.

As for $B$, note that if you change the indices in $(\cL_i)$ by an additive constant, you still get an elliptic helix of period 3 whose endomorphism algebra is equigenerated by $d$ elements. It thus suffices to show that
$$
H_B(t) := \sum_{i\geq0} \dim (\cL_0,\cL_i) t^i =  \frac{1-t^3}{1 - dt + dt^2 - t^3}
$$
Let $r_i = \rank \cL_i, d_i = \deg \cL_i$ be the numerical invariants of $\underline{\cL}$ as also defined in (\ref{eq:dsandrs}). As observed in the proof of Proposition~\ref{prop:egSeedWhichGen}, Proposition~\ref{prop:drRecursionRelations} together with the fact that the triad generating $\underline{\cL}$ and their mutations all have Hom dimension $(d,d,d)$, imply that the $r_i$'s and $d_i$'s satisfy the recursive relation
\begin{equation}  \label{eqn:HilbertRecursion}
    a_{i+3} - da_{i+2} + da_{i+1} -a_i = 0.
\end{equation}
Hence the same is true of
$$
h_i := \begin{vmatrix}
d_i & d_0 \\ r_i & r_0
\end{vmatrix}.
$$
Furthermore, we have $h_i = \dim(\cL_0,\cL_i)$ for $i>0$ by Lemma~\ref{lemma.rightmutate}, although $\dim (\cL_0,\cL_0) = 1 \neq 0 = h_0$. It follows that $H_B(t) -1 = \sum_{i\geq 1} h_i t^i$. Now
$$
\frac{1-t^3}{1 - dt + dt^2 - t^3} - 1 =  (dt-dt^2) (1 + dt + \ldots) =  dt + (d^2-d)t^2 + \ldots
$$
In view of Proposition~\ref{prop:KoszulResultionEquigen}(1) and the recursive relation the $h_i$ satisfy, we see this power series is just $\sum_{i\geq 1} h_i t^i$ and the proof is complete.
\end{proof}
In \cite{atv1} and \cite{bp}, the case in which $d=3$ is analyzed in detail. In that case, Artin-Tate-Van den Bergh show that the analogous graded ring $B$ is obtained from the Koszul Artin-Shelter regular algebra $A$ by factoring out a normal degree three element. Our Hilbert series calculation suggests a similar result in our case. We turn to proving this.

We recall from \cite[Section~4]{species}, the notion of a normal family of elements in an indexed algebra.
\begin{definition}  \label{def:normal}
A {\it normal family of elements of  degree $d$} in a $\bZ$-indexed algebra $A$ is a family $\underline{g} = (g_i)_{i \in \bZ}$ where $g_i \in A_{i,i+d}$ and such that for every $i,j$ we have $g_i A_{i+d,j+d} = A_{i,j} g_j$. An element $g \in A_{ij}$ is said to be {\it regular} if  right and left multiplication by $g$ are injective as maps from $Ae_i \to Ae_j$ and $e_jA \to e_iA$ respectively.
\end{definition}

\begin{theorem}  \label{thm:kernelGenByNormal}
Let $\underline{\cL}$ be an elliptic helix of period 3 in $\Coh(X)$ such that $B = \End(\underline{\cL})$ is equigenerated by $d$ elements. Let
$$
\pi \colon A := \bS^{nc}(\underline{\cL}) \to B
$$
be the natural surjection from Theorem \ref{thm:BPBisKoszul}(2). Then $\ker \pi$ is generated by a normal family of regular elements of degree 3.
\end{theorem}
\begin{proof}
We follow the proof of the analogous graded result in \cite[Section~7]{atv1}, the key difference being that we will replace the cohomological study of line bundles from \cite{atv1} by the cohomology of elliptic helices of period 3. From our Hilbert series calculation Proposition~\ref{prop:HilbertSeries}, we know that the kernel of the surjection $A_{i,i+3} \to B_{i,i+3}$ is one dimensional so we can pick a $k$-basis $g_i \in A_{i,i+3}$. We wish to show that $\underline{g}:=(g_i)$ is a normal family of regular elements generating $\ker \pi$, for which it suffices to show that $\ker \pi$ is generated on the left and on the right, in degree three.

Let $\underline{V} = (B_{i,i+1})_{i \in \bZ}$ be the space of degree one generators for $A$ and $B$ and $\phi \colon T(\underline{V}) \to B$ be the natural surjection from the tensor algebra $T(\underline{V})$. Let $J = \ker \phi$. We define the following kernels of multiplication maps
$$
K_{abc} := \ker (B_{ab} \otimes B_{bc} \to B_{ac}), \quad
K_{abcd} := \ker (B_{ab} \otimes B_{bc} \otimes B_{cd} \to B_{ad}) .
$$
The proof of the following result is elementary and identical to that found in \cite[Lemma~7.27]{atv1}
\begin{lemma}  \label{lem:atvJandK}
The natural surjection $T_{ad} \to B_{ab} \otimes B_{bc} \otimes B_{cd}$ induces an isomorphism
$$
\frac{J_{ad}}{T_{ab}\otimes J_{bd} + J_{ac} \otimes T_{cd}}   \simeq
\frac{K_{abcd}}{B_{ab}\otimes K_{bcd} + K_{abc} \otimes B_{cd}}
$$
\end{lemma}
The key technical lemma is the following.
\begin{lemma}  \label{lem:keyToNormal}
For $n>3$ we have
\begin{enumerate}
    \item $\dfrac{K_{012n}}{B_{01}\otimes K_{12n} + K_{012} \otimes B_{2n}} = 0$ and,
    \item $J_{i,i+n} = T_{i,i+1} J_{i+1,i+n} + J_{i,i+2} T_{i+2,i+n}$ for all $i \in \bZ$.
\end{enumerate}
\end{lemma}
\begin{proof}
Note that Lemma~\ref{lem:atvJandK} shows that part~(1) implies part~(2) in the special case where $i=0$. However, changing the indices in $\underline{\cL}$ by an additive constant we see that part~(2) holds for all $i$.

It thus suffices to prove part~(1), which we do presently. Consider the following commutative diagram with exact rows and columns
$$
\begin{CD}
@. B_{01} \otimes K_{12n} @= B_{01} \otimes K_{12n}  @. @. \\
@. @VVV @VVV @. @. \\
0 @>>> K_{012n} @>>> B_{01}\otimes B_{12}\otimes B_{2n} @>>> B_{0n} @>>> 0 \\
@. @VVV @VVV @VVV @. \\
0 @>>> \frac{K_{012n}}{B_{01} \otimes K_{12n}} @>>> B_{01} \otimes B_{1n} @>>> B_{0n} @>>> 0
\end{CD}
$$
which naturally produces an isomorphism
$$\frac{K_{012n}}{B_{01} \otimes K_{12n}} \simeq K_{01n}.
$$
We need the following
\begin{claim} \label{claim:kernelMult}
$K_{01n} \simeq (\cL_{-n}, L_{\cL_{-1}}\cL_0)$.
\end{claim}
\begin{proof}
This follows on applying the functor $(\cL_{-n}, ?)$ to the exact sequence
$$
0 \to L_{\cL_{-1}} \cL_0 \to (\cL_{-1},\cL_0) \otimes \cL_{-1} \to \cL_0 \to 0 .
$$
\end{proof}
We return to the proof of Lemma~\ref{lem:keyToNormal} for which it remains only to show that the cokernel of the map $K_{012} \otimes (\cL_{-n}, \cL_{-2}) \to K_{01n}$ is zero. In view of Claim~\ref{claim:kernelMult}, this amounts to showing the composition of morphisms map
$$
\mu \colon (\cL_{-2}, L_{\cL_{-1}} \cL_0) \otimes (\cL_{-n}, \cL_{-2}) \to (\cL_{-n}, L_{\cL_{-1}} \cL_0)
$$
is surjective. Applying the functor $(\cL_{-n}, ?)$ to the exact sequence
$$
0 \to \cL_{-3} \to (\cL_{-2}, L_{\cL_{-1}}\, \cL_0) \otimes \cL_{-2} \to L_{\cL_{-1}} \cL_0 \to 0
$$
gives the exact sequence
$$
(\cL_{-2}, L_{\cL_{-1}}\, \cL_0) \otimes (\cL_{-n}, \cL_{-2}) \xrightarrow{\mu}  (\cL_{-n}, L_{\cL_{-1}}\, \cL_0) \to \,^1(\cL_{-n}, \cL_{-3}).
$$
This completes the proof of the lemma since by definition of an elliptic helix of period 3, we have $^1(\cL_{-n}, \cL_{-3}) = 0$ as soon as $n>3$.
\end{proof}
We now return to completing the proof of Theorem~\ref{thm:kernelGenByNormal}. Recall $\underline{g} = (g_i)$ was defined to be the family of elements $g_i \in A_{i,i+3}$ which span the kernel of $\pi_{i,i+3} \colon A_{i,i+3} \to B_{i,i+3}$. We first show that $\ker \pi  = A\underline{g} := \oplus A e_i g_i$. Indeed, this follows by induction on degree using Lemma~\ref{lem:keyToNormal}(2) and the fact that $J_{i,i+2}$ is zero in $A$. To prove the right-handed result $\ker \pi = \underline{g}A$, we need only apply the left-handed result to the dual elliptic helix of period three, $ \ldots, \cL_2^*, \cL_1^*, \cL_0^*, \ldots$. Finally, regularity of the $g_i$ follows from our Hilbert series results Proposition~\ref{prop:HilbertSeries}.
\end{proof}

In the case of an elliptic helix generated by a $(0,\tfrac{d}{2},d)$-triad, we can say a little more about its endomorphism ring, or rather, its {\sf Proj}.
\begin{theorem}  \label{thm:ncElliptic}
Let $\underline{\mathcal{L}}$ be an elliptic helix of period 3 generated by a $(0,\tfrac{d}{2},d)$-triad where $d>3$ is odd. Let $B = \End(\cL)$ and $A = \bS^{nc}(\underline{\cL})$ be its quadratic part. Then
\begin{enumerate}


\item{} $A$ and $B$ are nonnoetherian and $B$ is coherent and,


\item{} ${\sf Proj }B$ is a noncommutative elliptic curve.

\end{enumerate}
\end{theorem}

\begin{proof}




In light of Theorem~\ref{thm:egTriadGenHelix}, the proof of \cite[Theorem 3.5]{polish1} shows that the sequence $\underline{\mathcal{L}}$ is ample for the noncommutative elliptic curve ${\sf C}_{\theta}$ defined in \cite{polish1}, where $\theta = d \biggl{(}\frac{A-(d-1)}{A-(d-3)}\biggr{)}$.  This implies that $B$ is coherent \cite[Proposition 2.3]{polishchuk}, and establishes (2).  The irrationality of $\theta$, proven in Theorem \ref{thm:egTriadGenHelix}, implies, by \cite[Proposition 3.1]{polish1} that every nonzero object of ${\sf C}_{\theta}$ is nonnoetherian, so that $B$ is nonnoetherian, i.e. $e_{i}B$ is nonnoetherian for all $i \in \mathbb{Z}$. Thus, $A$ is nonnoetherian, whence (1).

\end{proof}


\bibliographystyle{amsalpha}

\bibliography{main}

\end{document}